%% file: okounkovBodies.tex
\documentclass[a4paper]{amsart}

\usepackage{amsthm,amsmath,amssymb}
\usepackage{pstricks}
\usepackage{graphics}
\usepackage{color}
\usepackage[all]{xy}
\usepackage{subfigure}
\usepackage[american]{babel}
\theoremstyle{definition}
\newtheorem{theorem}{Theorem}[section]
\newtheorem{lemma}[theorem]{Lemma}
\newtheorem{proposition}[theorem]{Proposition}
\newtheorem{corollary}[theorem]{Corollary}
\newtheorem{definition}[theorem]{Definition}
\newtheorem{example}[theorem]{Example}
\newtheorem{remark}[theorem]{Remark}

\newtheorem*{introTheorem}{Theorem}

\newcommand{\cR}{\mathcal{R}}
\newcommand{\cD}{\mathcal{D}}
\newcommand{\cL}{\mathcal{L}}
\newcommand{\cO}{\mathcal{O}}
\newcommand{\cA}{\mathcal{A}}
\newcommand{\cP}{\mathcal{P}}
\newcommand{\cV}{\mathcal{V}}

\newcommand{\cX}{\mathcal{X}}

\newcommand{\cC}{\mathcal{C}}

\newcommand{\bK}{\mathbb{K}}
\newcommand{\bC}{\mathbb{C}}
\newcommand{\bR}{\mathbb{R}}
\newcommand{\bQ}{\mathbb{Q}}
\newcommand{\bZ}{\mathbb{Z}}
\newcommand{\bN}{\mathbb{N}}
\newcommand{\bP}{\mathbb{P}}
\newcommand{\bA}{\mathbb{A}}
\newcommand{\bS}{\mathbb{S}}
\newcommand{\bF}{\mathbb{F}}

\newcommand{\Cstar}{\mathbb{C}^*}

\DeclareMathOperator{\divisor}{div}

\DeclareMathOperator{\cadiv}{CaDiv}
\DeclareMathOperator{\cl}{Cl}

\DeclareMathOperator{\wdiv}{Div}
\DeclareMathOperator{\pic}{Pic}
\DeclareMathOperator{\spec}{Spec}
\DeclareMathOperator{\Spec}{\textbf{Spec}}
\DeclareMathOperator{\proj}{Proj}
\DeclareMathOperator{\ord}{ord}
\DeclareMathOperator{\orb}{orb}

\DeclareMathOperator{\relint}{relint}
\DeclareMathOperator{\degree}{deg\,}
\DeclareMathOperator{\id}{id}
\DeclareMathOperator{\coeff}{coeff}

\DeclareMathOperator{\conv}{conv}
\DeclareMathOperator{\cone}{cone}

\DeclareMathOperator{\Hom}{Hom}

\DeclareMathOperator{\Pol}{Pol}

\DeclareMathOperator{\TV}{TV}

\DeclareMathOperator{\tail}{tail}
\DeclareMathOperator{\minvert}{min_{v \in \fan_P}}

\DeclareMathOperator{\Loc}{Loc}
\newcommand{\fan}{\mathcal{S}}

\DeclareMathOperator{\tdiv}{T-Div}

\DeclareMathOperator{\SF}{SF}
\DeclareMathOperator{\CaSF}{CaSF}
\DeclareMathOperator{\tcadiv}{T-CaDiv}

\newcommand{\hstar}{h^*}

\DeclareMathOperator{\vol}{vol}
\DeclareMathOperator{\OB}{\Delta_{Y_\bullet}}
\DeclareMathOperator{\psEff}{\ovl{Eff}}
\DeclareMathOperator{\BigC}{Big}

\newcommand{\xfix}{x_{\textnormal{fix}}}
\newcommand{\sigmafix}{\sigma_{\textnormal{fix}}}
\newcommand{\deltafix}{\delta_{\textnormal{fix}}}

\newcommand{\ldef}{\; := \;}
\newcommand{\lst}{\,|\;}
\newcommand{\lra}{\longrightarrow}

\newcommand{\ovl}{\overline}
\newcommand{\wt}{\widetilde}

\newcommand{\val}[1]{\nu_{Y_{\bullet}, #1}}
\newcommand{\posOrthant}[1]{#1_{\geq 0}}

\newcommand{\ul}{\underline}

\renewcommand{\iff}{\Leftrightarrow}

\newcommand{\bangle}[1]{\langle #1 \rangle}
\newcommand{\floor}[1]{\lfloor #1 \rfloor}

\include{figuresOB}

\begin{document}

\title{Okounkov Bodies of Complexity-One $T$-Varieties}

\author[L.~Petersen]{Lars Petersen}
\address{Institut f\"ur Mathematik und Informatik,
         Freie Universit\"at Berlin
	 Arnimallee 3,
	 14195 Berlin, Germany}
\email{petersen@math.fu-berlin.de}


\begin{abstract}
We compute Okounkov bodies of projective complexity-one $T$-varie\-ties with respect to two types of invariant flags.
In particular, we show that the latter are rational polytopes. Moreover, using results of Dave Anderson and Nathan Ilten,
we briefly exhibit explicit links to degenerations and $T$-deformations. Finally, we prove that the global Okounkov body of a
rational projective complexity-one $T$-variety is rational polyhedral.
\end{abstract}

\maketitle

\section{Introduction}
\label{sec:introduction}

Toric varieties and projectivized rank two toric vector bundles over toric varieties are special instances of rational complexity-one
$T$-varieties. Fixing a torus invariant admissible flag, Robert Lazarsfeld and Mircea Musta\c{t}\u{a} not only provided a description
of Okounkov bodies of smooth projective toric varieties but also gave a presentation of the global Okounkov body,
cf.\ \cite[Section 6]{okounkovBody}. Jos\'{e} Gonz\'{a}lez showed that the global Okounkov body with respect to a suitable torus
invariant flag for projectivized rank two toric vector bundles over smooth projective toric varieties is a rational polyhedral cone.
In particular, he explicitly described the supporting hyperplanes of this cone in terms of Klyachko's filtration data of the bundle,
cf.\ \cite[Theorem 5.2]{obToricVB}.

The aim of this article is to extend these results to the class of rational projective complexity-one $T$-varieties. In order
to do so, we make use of their description in terms of so-called divisorial fans $\fan$ (see (\ref{subsec:complOne})).

Let $T = N \otimes \Cstar$ denote an algebraic torus that acts effectively on the rational projective complexity-one $T$-variety
$\TV(\fan)$. Depending on the local structure of $\TV(\fan)$ around a $T$-fixed point $\xfix$, we use the quotient representation of
$\TV(\fan)$ over the curve $C$ to construct two types (general and toric) of $T$-invariant admissible flags in $\TV(\fan)$. Furthermore,
given a divisorial support function $h$ and denoting by $D_h$ the corresponding $T$-invariant divisor (see (\ref{subsec:CD})), 
we can put Theorem \ref{thm:localOBgeneral}, Proposition \ref{prop:localOBtoricATwo} and Proposition \ref{prop:localOBtoricBTwo}
under one roof to obtain the following result:

\begin{introTheorem}
The Okounkov body $\OB(D_h)$ of a $T$-invariant big divisor $D_h$ on a rational projective complexity-one $T$-variety
$\TV(\fan)$ with respect to a general or toric flag $Y_\bullet$ is a rational polytope.
\end{introTheorem}

\noindent
More specifically, $\OB(D_h)$ can be described in terms of the piecewise affine linear function $\hstar$ over the polytope $\Box_h$
both of which are closely related to $h$. Moreover, using a toric downgrade, we obtain new results for Okounkov bodies of toric
varieties previously out of reach with purely toric methods (see Examples \ref{ex:hirzebruchDegenerationOB}, \ref{ex:nonSmoothTV}).

Having a precise description of the Okounkov bodies at hand, we proceed by investigating their relations to degenerations as constructed
by Dave Anderson (see \cite{obToricDegenerations}) and so-called one-parameter $T$-deformations which were introduced by Nathan Ilten
(see \cite{phdNathan}). Furthermore we give two examples for which we construct $T$-deformations via a decomposition of the respective Okounkov
body (see (\ref{subsec:ilten})).

Finally, we study the global Okounkov body of a rational projective complexity-one $T$-variety $\TV(\fan)$. Using the fact that this cone
can be constructed via the graph of a continuous piecewise affine linear function over some simpler cone, we obtain the following result:

\begin{introTheorem}
The global Okounkov body $\OB(\TV(\fan))$ of a rational projective com\-plexity-one $T$-variety $\TV(\fan)$ with respect to a general or
toric flag $Y_\bullet$ is a rational polyhedral cone.
\end{introTheorem}

The present paper is organized as follows. Section \ref{sec:polDiv} recalls the necessary ingredients from the theory of complexity-one
$T$-varieties: polyhedral divisors, marked fansy divisors, divisorial polytopes, and the description of torus invariant divisors in terms
of divisorial support functions.

Section \ref{sec:OB} forms the heart of this article. After recalling the general construction of Okounkov bodies and their description
in the toric setting, we construct two types of admissible invariant flags for complexity-one $T$-varieties and compute Okounkov bodies
with respect to these flags in terms of divisorial support functions. In particular, we derive that these convex bodies
are rational polytopes.

Section \ref{sec:ToricDeg} reviews the relation of Okounkov bodies to deformations and degenerations, as presented by Nathan Ilten
\cite{phdNathan} and Dave Anderson \cite{obToricDegenerations}. Several examples are used to exhibit this interplay and to illustrate
the combinatorial flavor which can be associated to the respective geometric operation.

Section \ref{sec:GlobalOB} is concerned with global Okounkov bodies of rational projective complexi\-ty-one $T$-varieties and provides
a proof that the latter are rational polyhedral.\\[-1ex]

\noindent
\emph{Acknowledgements.} I would like to thank Klaus Altmann, Christian Haase, Andreas Hochenegger, Nathan Ilten, Priska Jahnke,
Hendrik S\"u\ss{} and Jarek Wi\'{s}niewski for valuable comments and fruitful discussions.

\section{Preliminaries}
\label{sec:polDiv}

This section fixes some notation and introduces the language of polyhedral divisors and divisorial fans with a special focus upon
complexity-one $T$-varieties (see \cite{tvar1,tvar2} and in particular \cite{tvarReview} for more details). Moreover, we review the
crucial notions and main results of \cite{tidiv} and conclude with several examples which will also reappear in later sections.

For the rest of the paper we adopt the following conventions and notation. If not stated otherwise
\begin{itemize}
\item[--] $\bK$ denotes an algebraically closed field of characteristic zero.
\item[--] a \emph{variety} means an integral, separated scheme of finite type over $\bK$.
\item[--] $N$ denotes a lattice, i.e.\ a free abelian group of finite rank. Its dual $\Hom_\bZ(N,\bZ)$ is usually denoted by $M$.
Given a lattice $L$,  we set $L_\bQ := L \otimes_\bZ \bQ$ and $L_\bR := L \otimes_\bZ \bR$.
\item[--] a \emph{cone} is supposed to be pointed and polyhedral.
\item[--] we call a real-valued function $f: K \to \bR$ defined over some convex subset $K \subset \bR^k$ \emph{concave} if
$f(tx_1 + (1-t)x_2) \geq tf(x_1) + (1-t)f(x_2)$ for all $x_1, x_2 \in K$ and $0 \leq t \leq 1$.
\end{itemize}

\subsection{Polyhedral Divisors and Divisorial Fans}
\label{subsec:generalTheory}

\begin{definition}
A \emph{$T$-variety} is a normal variety $X$ together with an effective algebraic torus action $T \times X \to X$. Its
\emph{complexity} is defined as the codimension of a generic $T$-orbit. 
\end{definition}

The most prominent and best understood $T$-varieties are those of complexity zero. They can be described via the combinatorial language
of \emph{polyhedral fans} and are much better known under the name of \emph{toric varieties}.

Let $T$ be a $k$-dimensional affine torus ($\cong (\Cstar)^k$), and let $M$, $N$ denote the mutually dual, free abelian groups
($\cong \bZ^k$) of characters and one-parameter subgroups, respectively. In particular, $T$ can be recovered as
$T= \spec \bC[M] = N \otimes_\bZ \bC^*$. For a polyhedral cone $\sigma \subseteq N_\bQ$ we may consider the semigroup (with respect
to Minkowski addition)
\[\Pol^+_\bQ(N,\sigma) := \{\Delta\subseteq N_\bQ \lst \Delta = \mbox{polyhedron with} \, \tail \Delta = \sigma\}
\subseteq \Pol_\bQ(N,\sigma)\]
where $\tail \Delta := \{a \in N_\bQ \lst \Delta + a \subseteq \Delta\}$ denotes the \emph{tailcone} of $\Delta$ and
$\Pol \supseteq \Pol^+$ is the associated Grothendieck group.

On the other hand, let $Y$ be a normal and semiprojective variety, i.e.\ $Y\to Y_0$ is projective over an affine $Y_0$.
By $\cadiv(Y)$ we denote the group of Cartier divisors on $Y$. A $\bQ$-Cartier divisor on $Y$ is called \emph{semiample} if it has
a positive, base point free multiple. For an element 
\[\cD = \sum_i \cD_i \otimes D_i \in \Pol_\bQ (N,\sigma) \otimes_\bZ \cadiv(Y)\]
with $\cD_i \in \Pol^+(N_\bQ,\sigma)$ and effective divisors $D_i$, we may consider its evaluations
\[\cD(u) := \sum_i \min \langle \cD_i,u \rangle \, D_i \in \cadiv_\bQ(Y)\]
on elements $u \in \sigma^\vee$. Since we always want to assume that $Y$ is projective, we will explicitly allow $\emptyset \in
\Pol^+_\bQ(N,\tail \cD)$ as polyhedral coefficients of $\cD$. Then, $\cD = \sum_i \cD_i \otimes D_i$ should be interpreted as
$\sum_{\cD_i \neq \emptyset} \cD_i \otimes D_i|_{\Loc \cD}$ with $\Loc \cD := Y \setminus \bigcup_{\cD_i = \emptyset} D_i$.

We call $\cD$ a \emph{p-divisor} if the $\cD(u)$ are semiample and, moreover, big for $u \in \relint \sigma^\vee$. The common
tailcone $\sigma$ of the coefficients $\cD_i$ will be denoted by $\tail \cD$. The positivity assumptions imply that
$\cD(u)+\cD(u') \leq \cD(u+u')$, so that $\cA(\cD) := \bigoplus_{u \in \sigma^\vee \cap M} \cO_{\Loc \cD}(\cD(u))$ becomes a sheaf of
rings. We then define the following two $T$-varieties
\[\wt{\TV}(\cD):= \Spec_{\Loc \cD} \cA(\cD) \quad \textnormal{and} \quad \TV(\cD) := \spec \Gamma(\Loc \cD,\cA(\cD))\,.\]
Note that this construction comes with a natural \emph{contraction} map $r: \wt{\TV}(\cD) \to \TV(\cD)$ which identifies certain
$T$-orbits in $\wt{\TV}(\cD)$ and, furthermore, a \emph{quotient} map $\pi: \wt{\TV}(\cD) \to \Loc \cD$.

$\TV(\cD)$ does not change if $\cD$ is pulled back via a birational modification $Y'\to Y$ or if $\cD$ is modified by a
polyhedral \emph{principal} divisor on $Y$ where the latter denotes an element in the image of the natural map
$N \otimes_\bZ \bK(Y)^* \to \Pol_\bQ(N,\sigma) \otimes_\bZ \cadiv(Y)$. Two p-divisors that differ by chains of the upper operations
are called equivalent. 

\begin{theorem}
\label{thm:affCorr}
Cf.\ \cite[Theorems 3.1 and 3.4; Corollary 8.12]{tvar1}. The map $\cD \mapsto \TV(\cD)$ yields a bijection between equivalence classes
of p-divisors and normal, affine varieties with an effective torus action.
\end{theorem}

\subsection{Complexity-One $T$-Varieties}
\label{subsec:complOne}
Having established the theory of affine $T$-varieties for arbitrary complexity in the previous section, we will from now on restrict to
\emph{complexity one} $(k = d-1)$. This restriction comes with a lot of technical simplifications, because $Y =: C$ becomes a
smooth projective curve. For example, introducing the addition rule $\Delta + \emptyset := \emptyset$, we may define the degree of
$\cD = \sum_i \cD_i \otimes D_i$ as 
\[\deg \cD := \sum_i (\deg D_i) \cdot \cD_i \in \Pol^+_\bQ(N,\tail \cD).\]
In particular, $\deg \cD = \emptyset \iff  \Loc \cD \neq C \iff \Loc \cD$ is affine. This easily implies the following criterion:
$\cD$ is a p-divisor if and only if $\deg \cD \subsetneq \tail \cD$ and, additionally, $\cD((\gg 0)\cdot w)$ is principal for
$w \in (\tail \cD)^\vee$ with $w^\perp \cap (\deg \cD)\neq 0$. Note that the latter condition is automatically fulfilled if $C = \bP^1$.

The assignment $\cD \mapsto \TV(\cD)$ from Theorem \ref{thm:affCorr} is functorial. In particular, as was shown in \cite{polTvar},
if $\cD$ is a p-divisor containing some $\cD'= \sum_i \cD_i' \otimes D_i$ (meaning that $\cD'_i \subseteq \cD_i$ for all $i$), then
$\cD'$ is again a p-divisor which induces a $T$-equivariant open embedding $\TV(\cD') \hookrightarrow \TV(\cD)$ if and only if
$\cD' \leq \cD$, i.e.\ if all coefficients $\cD_i' \leq \cD_i$ are faces, and $\deg \cD' = \deg \cD \cap \tail \cD'$.

In particular, if p-divisors $\cD^\nu$ are arranged in a so-called \emph{divisorial fan} $\fan = \{\cD^\nu\}$, then we can glue the
associated affine $T$-varieties $\TV(\cD^\nu)$ to obtain a separated $\TV(\fan) = \bigcup_\nu \TV(\cD^\nu)$, cf.\ \cite{tvar2}. The
\emph{slices} $\fan_i = \{\cD^\nu_i\}$ form a polyhedral subdivision in $N_\bQ$, and $\tail \fan := \{\tail \cD_i\}$ is
called the \emph{tailfan} of $\fan$.

Moreover, the subsets $\deg \cD \subsetneq \tail \cD$ glue to a subset $\deg \fan \subsetneq |\tail \fan|\subseteq N_\bQ$. Roughly
speaking, we understand that $\fan = \sum_i \fan_i\otimes D_i$. Yet, to keep the full information of the divisorial fan $\fan$ one
needs a labeling of the $\fan_i$-cells indicating the p-divisor they come from. However, following \cite{polTvar}, the technical
description can be reduced considerably since one may eventually forget about the labeling. Instead, one only needs to mark those
cones $\tail \cD_i$ inside $\tail \fan$ which have non-empty $\deg \cD_i$. The marked cones together with the formal sum
$\fan = \sum_i \fan_i \otimes D_i$ then yield a \emph{marked fansy divisor} associated with the divisorial fan $\fan$,
cf.\ (\ref{subsec:marking}).

\subsection{Toric Downgrades}
\label{subsec:toricDown}
A handy technique to generate instructive examples is to consider toric varieties with an effective subtorus action
(cf.\ \cite[Section 5]{tvar2}). Let $\TV(\Sigma)$ be a (complete) $d$-dimensional toric variety with embedded torus $T_N$. Fixing
a subtorus $T_{N'} \hookrightarrow T_N$ then corresponds to an exact sequence of lattices 
\[\xymatrix{ 0 \ar[r] & N' \ar[r]^{F} & N \ar[r]^{P} &  N'' \ar[r] &  0 } \,.\] 
By choosing a cosection $s: N \to N'$, we induce a splitting $N \cong N' \oplus N''$ with projections $s:\, N \rightarrow N'$, and
$P:\, N \rightarrow N''$. Set $Y \ldef \TV(\Sigma')$, where $\Sigma'$ is an arbitrary smooth projective fan $\Sigma'$ refining
the images $P(\delta)$ of all cones $\delta \in \Sigma$. Then every cone $\sigma \in \Sigma(d)$ gives rise to a p-divisor
$\cD^{\sigma}$. Namely, for each ray $\rho' \in \Sigma'(1)$, let $n_{\rho'}$ denote its primitive generator and set
\[\cD_{\rho'}(\sigma) = s_\bQ \big( P_\bQ^{-1}(n_{\rho'})\cap \sigma \big) \quad \textnormal{ and } \quad  \cD^\sigma =
\sum_{\rho' \in \Sigma'(1)} \cD_{\rho'}(\sigma) \otimes D_{\rho'}.\] 
Finally, $\{\cD^\sigma\}_{\sigma \in \Sigma(d)}$ is a divisorial fan. Observe that for certain p-divisors $\cD^\sigma$ and rays
$\rho' \in \Sigma'(1)$ the intersection $P_\bQ^{-1}(n_{\rho'}) \cap \sigma$ may be empty. In this case we have that
$\cD_{\rho'}(\sigma) = \emptyset$.

\begin{example}
\label{ex:hirzebruch}
We consider the $n$'th Hirzebruch surface $\bF_n$ as a $\bK^*$-surface via the following maps of lattices 
\[F = \left(\begin{array}{c} 1 \\ 0 \end{array} \right)\,, \qquad  P = \left(\begin{array}{cc} 0 & 1 \end{array}\right)\,, \qquad
 s = \left(\begin{array}{cc} 1 & 0 \end{array}\right)\,.\]
The slices of the divisorial fan $\fan$ which arise from this downgrade are illustrated in Figure \ref{fig:TDhirzebruch}. More
specifically, we have that
\[\begin{array}{ll} \cD^{\sigma_0} = [0,\infty) \otimes [0] + \emptyset \otimes [\infty], & \cD^{\sigma_1} = [-1/n \;\; 0] \otimes [0]
+ \emptyset \otimes [\infty]\,, \\[1mm] \cD^{\sigma_2} = (-\infty \;\, -1/n] \otimes [0]\,,  & \cD^{\sigma_3} = \emptyset \otimes [0]
+ [0 \;\; \infty) \otimes [\infty]\,. \end{array}\]
\end{example}

\begin{figure}[h]
\centering
\TDhirzebruch
\caption{Divisorial fan associated to $\bF_n$, cf.\ Example \ref{ex:hirzebruch}.}
\label{fig:TDhirzebruch}
\end{figure}

\subsection{Marked Fansy Divisors}
\label{subsec:marking}
A \emph{marked fansy divisor} on a curve $C$ is a formal sum 
\[\Xi = \sum_{P \in C} \Xi_P \otimes [P]\]
together with a complete fan $\Sigma \subset N_\bQ$ and a subset $\cC \subset \Sigma$ such that 

\begin{enumerate}
\item for all $P \in Y$, the coefficient $\Xi_P$ is a complete polyhedral subdivision of $N_\bQ$ with $\tail \Xi_P = \Sigma$.
\item for a cone $\sigma \in \cC$ of full dimension the p-divisor $\cD^\sigma = \sum_P \cD^\sigma_P \otimes [P]$ is proper where
$\cD^\sigma_P$ denotes the unique polyhedron in $\Xi_P$ whose tailcone is equal to $\sigma$.
\item for a full dimensional cone $\sigma \in \cC$ and a face $\tau \prec \sigma$ we have that $\tau \in \cC$ if and only if $\deg \cD^\sigma \cap \tau \neq \emptyset$.
\item if $\tau$ is a face of $\sigma$ then $\tau \in \cC$ implies that $\sigma \in \cC$.
\end{enumerate}

The elements of $\cC \subset \Sigma$ are called \emph{marked} cones. Given a marked fansy divisor $\Xi$ on the curve $C$ one can
construct a complete divisorial fan $\fan$ with $\Xi(\fan) = \Xi$. Moreover, two divisorial fans $\fan_1,\fan_2$ with
$\Xi(\fan_1) = \Xi(\fan_2)$ yield the same $T$-variety $\TV(\fan_1) = \TV(\fan_2)$, cf.\ \cite[Proposition 1.6]{polTvar}.

Conversely, one can easily associate a marked fansy divisor to a given divisorial fan $\fan$ by setting
\[\Xi(\fan) \ldef \sum_P \fan_P \otimes [P]\,, \quad \textnormal { and } \quad \cC(\fan) \ldef \big\{ \tail \cD \;|\; \cD \in \fan,
\, \Loc \cD = C \big\}\,. \]
Note that the elements of $\cC(\fan)$ capture the information of which orbits are identified via the contraction
$r: \wt{\TV}(\fan) \to \TV(\fan)$.

\begin{example}
\label{ex:hirzebruchFD}
The marked fansy divisor for $\bF_n$ as depicted in Example \ref{ex:hirzebruch} consists of the following data:
\[\Xi_0 = \fan_0, \quad \Xi_\infty = \fan_\infty, \quad \Sigma = \tail \fan, \quad \cC = \{(-\infty,0]\}\,.\]
\end{example}

\subsection{Divisorial Polytopes}
\label{subsec:divisorialPolytopes}
Following \cite{polTvar}, we also briefly recall the description of polarized complexity-one $T$-varieties in terms of divisorial
polytopes. This correspondence is a generalization of the relation between polarized projective toric varieties and lattice polytopes.

A \emph{divisorial polytope} $(\Psi,\Box,C)$ consists of a lattice polytope $\Box \subset M_\bQ$, a smooth projective curve $C$, and
a map
\[\Psi = \sum_{P \in C} \Psi_P \otimes [P]: \;\; \Box \longrightarrow \cadiv_\bQ C \,,\]
with concave piecewise affine linear ``coordinate'' functions $\Psi_P: \Box \to \bQ$ such that 
\begin{enumerate}
\item for all but finitely many $P \in C$ we have that $\Psi_P \equiv 0$.
\item $\degree \Psi(u) > 0$ for $u$ in the interior of $\Box$;
\item for $u$ a vertex of $\Box$, $\degree \Psi(u) >0$ or $\Psi(u) \sim 0$;
\item for all $P \in C$ the graph of $\Psi_P$ is integral, i.e.\ its vertices lie in $M \times \bZ$.
\end{enumerate}

See \cite[Section 3]{polTvar} for a description how to construct a marked fansy divisor $\Xi(\Psi)$ from a triple $(\Psi,\Box,C)$. The
crucial fact about this construction is that it provides us with a one-to-one correspondence between divisorial polytopes and pairs
$(X,\cL)$ of complexity-one $T$ varieties $X$ with an equivariant ample line bundle $\cL$ via the map
\[(\Psi,\Box,C) \mapsto \big(\TV\big(\Xi(\Psi)\big),\cO(D_{\Psi^*})\big),\] 
cf.\ \cite[Theorem 3.2]{polTvar}.

\subsection{Invariant Weil Divisors}
\label{subsec:WD}
For more details on the current subject we refer the reader to \cite[Section 3]{tidiv}. Let $\cD$ be a p-divisor on the curve $C$.
The elements of the set of invariant prime divisors in $\wt{\TV}(\cD)$ split into two different types. On the one hand, we have the
so-called \emph{vertical} divisors
\[D_{(P,v)} := \ovl{\orb}(P,v).\]
They are associated with prime divisors $P \in C$ together with a vertex $v \in \cD_P$. On the other hand, there are the so-called
\emph{horizontal} divisors
\[D_\rho := \ovl{\orb}(\eta(C),\rho).\]
These correspond to rays $\rho$ of the cone $\tail \cD = \cD_{\eta(C)}$, where $\eta(C)$ denotes the generic point of $C$. The
invariant prime divisors in $\TV(\cD)$ then correspond exactly to those on $\wt{\TV}(\cD)$ which are not contracted via $r: \wt{\TV}(\cD)
\to \TV(\cD)$. Finally, we denote the free abelian group of $T$-equivariant Weil divisors in $\TV(\cD)$ by $\tdiv \TV(\cD)$. 

Note that the vertical divisors $D_{(P,v)}$ (with $P\in C(\bK)$ and $v \in \cD_P$) survive completely in $\TV(\cD)$ whereas
$D_\rho$ becomes contracted if and only if the ray $\rho$ is not disjoint from $\deg \cD$.

\begin{definition}
\label{def:extremalRay}
Let $\cD$ be a p-divisor on the curve $C$. The set of rays $\rho \in (\tail \cD)(1)$ with $\deg \cD \cap \rho = \emptyset$ is denoted
by $\cR(\cD)$. Given a divisorial fan $\fan$ on $C$, we set $\cR(\fan) \ldef \{\cR(\cD)\,|\, \cD \in \fan\}$.
\end{definition}

\subsection{Divisor Classes}
\label{subsec:divCL}
We keep the notation from (\ref{subsec:WD}) and denote by $\bK(C)$ the function field of $C$. The ring of semi-invariant, i.e.\
$M$-homogeneous rational functions on $\TV(\fan)$ is then given by $\bK(C)[M]$. For $v \in N_\bQ$
let $\mu(v)$ be the smallest integer $k \geq 1$ such that $k\cdot v$ is a lattice point and denote by $n_\rho$ the primitive generator
of the ray $\rho$. According to \cite[Proposition 3.14]{tidiv} the principal divisor associated to $f \chi^u \in \bK(C)[M]$ on
$\wt{\TV}(\cD)$ or $\TV(\cD)$ is then given by
\[\divisor \big(f \chi^u \big) = \sum_{\rho} \langle n_\rho,u \rangle D_\rho + \sum_{(P,v)} \mu(v) \big(\langle v,u\rangle +
\ord_P f \big) D_{(P,v)}\]
where, if focused on $\TV(\cD)$, one is again supposed to omit all prime divisors being contracted.

Let $\fan = \sum_{P \in \bP^1} \fan_P \otimes [P]$ be a complete divisorial fan on $C = \bP^1$. In particular,
$\deg \fan \subsetneq |\tail \fan| = N_\bQ$. Choose a non-empty finite set of points $\cP \subseteq \bP^1$ such that $\fan_p$ is
trivial (i.e.\ $\fan_P = \tail \fan$) for $P \in \bP^1 \setminus \cP$. For a vertex $v$ from some slice of $\fan$ we denote by
$P(v) \in \bP^1$ the point which is associated to the slice we have taken $v$ from. Let us furthermore define
$\cV \ldef \{v \in \fan_{P}(0) \lst P \in \cP\}$ and the following two natural maps
\[Q: \bZ^{(\cV \cup \cR)} \to \bZ^\cP/\bZ \quad \textnormal{with} \quad
e(v) \mapsto \mu(v)\,\ovl{e(p(v))} \quad \textnormal{and} \quad
e(\rho) \mapsto 0\,,\]
and
\vspace{-1ex}
\[\phi: \bZ^{(\cV \cup \cR)} \to N \quad \textnormal{with} \quad e(v) \mapsto \mu(v)v \quad \textnormal{and} \quad e(\rho)\mapsto
n_\rho \]
with $e(v)$ and $e(\rho)$ denoting the natural basis vectors. Deducing from (\ref{subsec:WD}) that
$\big(\bZ^{(\cV \cup \cR)}\big)^\vee \subset \tdiv (\TV(\fan))$, the discussion from above implies that one has the following
exact sequence describing the class group $\cl(\TV(\fan))$ of a complete rational complexity-one $T$ variety $\TV(\fan)$ (see also
\cite[Corollary 2.3]{tcox}):
\[0 \to (\bZ^\cP/\bZ)^\vee \oplus M \to \big(\bZ^{(\cV \cup \cR)}\big)^\vee \to \cl(\TV(\fan))\to 0 \]
where the first map is induced from $(Q,\phi)$.

With a view towards upcoming examples we finally recall that the canonical class of $\TV(\fan)$ can be represented as
\[K_{\TV(\fan)} = - \sum_{\rho \in \cR(\fan)} D_{\rho} + \sum_{(P,v)} \big(\mu(v) \coeff_P(K_C) + \mu(v) - 1 \big) D_{(P,v)} \,,\]
where $K_C$ is a representative of the canonical divisor on $C$ (cf.\ \cite[Theorem 3.21]{tidiv}).

\subsection{Invariant Cartier Divisors}
\label{subsec:CD}
Let $\fan$ be a divisorial fan on the curve $C$. A \emph{divisorial support function} on $\fan$ is a collection $(h_P)_{P \in C}$
of continuous piecewise affine linear functions $h_P: |\fan_P| \to \bQ$ such that 

\begin{enumerate}
\item $h_P$ has integral slope and integral translation on every polyhedron in the polyhedral complex $\fan_P \subset N_\bQ$.
\item all $h_P$ have the same linear part $=: \ul{h}$.
\item the set of points $P \in C$ for which $h_P$ differs from $\ul{h}$ is finite. 
\end{enumerate}

Observe that we may restrict an element $h_P \in \SF(\fan_P)$ to a subcomplex of $\fan_P$. More generally, we may restrict a
divisorial support function $h \in \SF(\fan)$ to a p-divisor $\cD \in \fan$. The latter restriction will be denoted by $h|_\cD$.

In addition, we can associate a divisorial support function $\SF(D)$ to any Cartier divisor $D \in \cadiv(C)$ by setting
$\SF(D)_P \equiv \coeff_P(D)$. Moreover, we can consider any element $u \in M$ as a divisorial support function by setting
$\SF(u)_P \equiv u$.

\begin{definition}
\label{def:CaSF}
A divisorial support function $h \in \SF(\fan)$ is called \emph{principal} if $h =  \SF(u) + \SF(D)$ for some $u \in M$ and some
principal divisor $D$ on $C$. It is called \emph{Cartier} if its restriction $h|_\cD$ is principal for every $\cD \in \fan$ with
$\Loc \cD = C$. The set of divisorial Cartier support functions is a free abelian group which we denote by $\CaSF(\fan)$.
\end{definition}

Let $\TV(\fan)$ be a complexity-one $T$-variety and denote by $\tcadiv(\TV(\fan))$ the free abelian group of $T$-invariant Cartier
divisors on $\TV(\fan)$. Proposition 3.10 from \cite{tidiv} states that $\tcadiv(\TV(\fan)) \cong \CaSF(\fan)$ as free abelian groups.
Thus, we will often identify an element $h \in \CaSF(\fan)$ with its induced $T$-invariant Cartier divisor $D_h$ via this
correspondence. Finally, we recall that the Weil divisor associated to a given Cartier divisor $h = (h_P)_P$ on $\TV(\fan)$ is equal to
\[-\sum_{\rho} \ul{h} (n_\rho) D_\rho - \sum_{(P,v)} \mu(v) h_P(v) D_{(P,v)}.\]
%

\subsection{Global Sections}
\label{subsec:globalSections}
Consider an invariant Cartier divisor $D_h$ together with its associated equivariant line bundle $\cO(D_h)$ on $\TV(\fan)$. Due to the
torus action we have an $M$-module structure on the $\bK$-vector space of global sections which decomposes into homogeneous summands
with respect to the elements of the character lattice
\[\Gamma\big(\TV(\fan),\cO(D_h)\big) = \bigoplus_{u\in M} \Gamma\big(\TV(\fan),\cO(D_h)\big)_u\,.\]
The so-called \emph{set of weights} of $D_h$ is defined as 
\[W(h) \ldef W(D_h) \ldef \{u \in M \mid \Gamma\big(\TV(\fan),\cO(D_h)\big)_u \neq 0\}\,.\]
We will now see how to bound $W(h)$ by a polyhedron which is defined via $h \in \CaSF(\fan)$ and how to describe the homogeneous
sections of a fixed weight $u \in M$ in terms of rational functions on the curve $C$.

\begin{definition}
\label{def:weightPolytopeCartier}
Given a Cartier support function $h = (h_P)_P$ on $\fan$ with linear part $\ul{h}$ we define its associated \emph{weight polyhedron}
as
\[\Box_h \def \Box_{\ul{h}} \ldef \{u \in M \mid  \langle u, v \rangle  \geq  \ul{h}(v) \quad\textnormal{for all}\quad v \in N\}\,.\]
In addition, we define the map $h^*:\Box_h \rightarrow \wdiv_\bQ C$ by 
\[h^*(u) \ldef \sum_P h_P^*(u)P \ldef \sum_P \minvert(u-h_P)P,\]
where $\minvert(u - h_P)$ denotes the minimal value of the continuous piecewise linear function $u - h_P$ along the vertices of
$\fan_P$. 
\end{definition}

The weight polyhedron captures the restriction of $h$ to the generic fiber $\TV(\tail \fan)$. It is compact if and only if $\fan$ is
complete, and its tailcone is given as the intersection of the dual cones of the elements in $\tail \fan$. Given a Cartier divisor
$D_h \in \tcadiv(\fan)$ with linear part $\ul{h}$ we have the following description of its global sections (cf.\
\cite[Proposition 3.23]{tidiv}):
\begin{enumerate}
\item $W(h)$ is a subset of $\Box_{h}$.
\item For a character $u \in \Box_{h}$ we have that
\[\Gamma\big(\TV(\fan),\cO(D_h)\big)_u = \Gamma\big(\Loc \fan,\cO_{\Loc \fan}(h^*(u))\big).\]
\end{enumerate}

\subsection{Examples}
\label{subsec:exGlobalSection}

\begin{figure}[b]
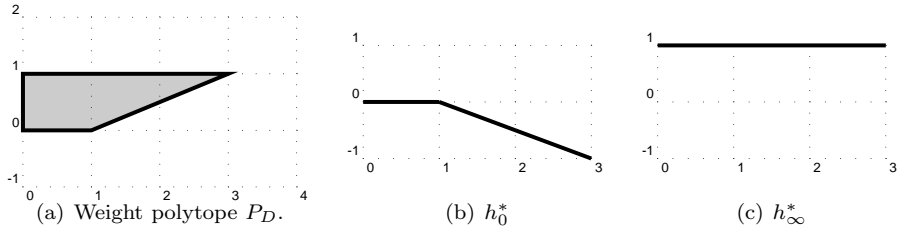

\centering
\subfigure[Weight polytope $P_D$.]{\hirzebruchTwoPolytope}\hspace*{5ex}
\subfigure[$\hstar_0$]{\hstarzero}\hspace*{5ex}
\subfigure[$\hstar_\infty$]{\hstarinfty}
\caption{Weight polytope and its divisorial analogue of the very ample line bundle $\cO(D)$ on $\bF_2$, cf.\ Example
\ref{ex:hirzebruch2LB}.}
\label{fig:hirzebruchData}
\end{figure}

\begin{example}
\label{ex:hirzebruch2LB}
We consider the downgrade of the second Hirzebruch surface $\bF_2$ as described in Example \ref{ex:hirzebruch} together with the
line bundle $\cL = \cO(D)$ given through the following generators of the global sections over the affine charts $U_{\sigma_i}$:
\[u_{\sigma_0} = [0 \; 0]\,,\quad u_{\sigma_1} = [1 \; 0]\,,\quad u_{\sigma_2} = [3 \; 1]\,,\quad u_{\sigma_3} = [0 \; 1]\,.\] 
Note that $\cL$ is very ample and defines an embedding into $\bP^5$. One can describe the embedding by a polytope $P_D \subset M_{\bQ}
= \bQ^2$ which is the convex hull of the $u_{\sigma_i}$. It contains six lattice points which form a basis of the $\bK$-vector space
$\Gamma(\bF_2,\cL)$, cf.\ Figure \ref{fig:hirzebruchData}(a). Using the toric downgrade construction from (\ref{subsec:toricDown}),
one finds that $\cL = \cO(D)$ with 
\[D = D_{([0],-1/2)} + D_{([\infty],0)}.\] 
By (\ref{subsec:globalSections}), we have $\Box_h = \{u \in \bZ \, | \; 3 \geq u \geq 0 \}$, and 
\[\begin{array}{ll} \Gamma(\bF_2,D_h)_0 = \Gamma \big(\bP^1,\cO([\infty])\big)\,, & \Gamma(\bF_2,D_h)_1 =
\Gamma \big(\bP^1,\cO([\infty])\big)\,, \\ \Gamma(\bF_2,D_h)_2 = \Gamma \big(\bP^1,\cO([\infty]- 1/2[0])\big)\,, &
\Gamma(\bF_2,D_h)_3 = \Gamma \big(\bP^1,\cO([\infty]-[0])\big)\,. \end {array} \]
On the whole, they sum up to a six dimensional vector space as expected from the toric picture. The corresponding graphs of $\hstar_0$
and $\hstar_\infty$ are shown in Figure \ref{fig:hirzebruchData}(b)+(c).
\end{example}

\begin{figure}[h]
\centering
\subfigure[$\fan_1$]{\quadricFanOne}\hspace*{2ex}
\subfigure[$\fan_0$]{\quadricFanZero}\hspace*{-2ex}
\subfigure[$\fan_\infty$]{\quadricFanInfty}
\caption{Non-trivial slices of $\fan(Q)$, cf.\ Example \ref{ex:quadric}.}
\label{fig:quadricFan}
\end{figure}

\begin{example}
\label{ex:quadric} 
The divisorial fan associated to the smooth quadric $Q = \TV(\fan)$ in $\bP^4$ as presented in \cite[Example 1.10]{candiv} is given in
Figure \ref{fig:quadricFan}. Note that $Q$ is Fano, i.e.\ $-K_Q$ is ample. The associated tailfan $\Sigma$ and degree $\deg \fan$ are
given in Figure \ref{fig:quadricFanPlus}. All maximal p-divisors have complete locus, i.e.\ $\cR = \emptyset$ and all rays in the
tailfan are marked.

\begin{figure}[h]
\centering
\subfigure[$\Sigma = \tail \fan$]{\tailfanQuadric}
\hspace*{8ex}
\subfigure[$\deg \fan$]{\degreeQuadric}
\caption{Tailfan and degree of $\fan(Q)$, cf.\ Example \ref{ex:quadric}.}
\label{fig:quadricFanPlus}
\end{figure}

Let us consider the the anti-canonical divisor which may be represented as $3D_{[\infty],(1/2,1/2)}$. The weight polytope $\Box_h$
associated to the corresponding support function $h$ is given in Figure \ref{fig:boxes}(a).
\end{example}

\begin{figure}[h]
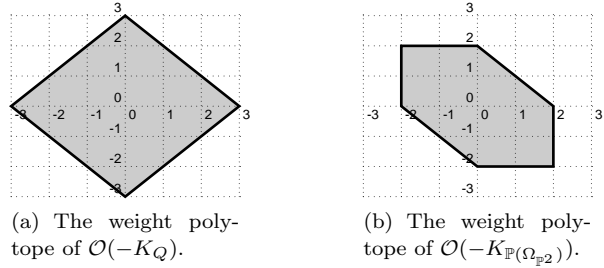

\centering
\subfigure[The weight polytope of $\cO(-K_Q)$.]{\boxQuadric}\hspace*{10ex}
\subfigure[The weight polytope of $\cO(-K_{\bP(\Omega_{\bP^2})})$.]{\boxCotang}
\caption{The weight polytopes from Examples \ref{ex:quadric} and \ref{ex:cotang}.}
\label{fig:boxes}
\end{figure}

\begin{example}
\label{ex:cotang}
We consider the smooth projective complexity-one Fano $T$-three\-fold $\bP(\Omega_{\bP^2})$ from \cite[section 8.5]{tvar2}.
The non-trivial slices of its divisorial fan $\fan$ over $\bP^1$ are illustrated in Figure \ref{fig:cotangFan}, whereas
$\Sigma = \tail \fan$ and $\deg \fan$ are given in Figure \ref{fig:cotangPlus}. As in the previous example, all maximal p-divisors
have complete locus, i.e.\ $\cR = \emptyset$ and all rays in the tailfan are marked.

\begin{figure}[h]
\centering
\subfigure[$\fan_0$]{\cotangfannullvv}\hspace*{5ex}
\subfigure[$\fan_1$]{\cotangfanonevv}\hspace*{5ex}
\subfigure[$\fan_\infty$]{\cotangfaninftyvv}
\caption{Non-trivial slices of $\fan(\bP(\Omega_{\bP^2}))$, cf.\ Example \ref{ex:cotang}.}
\label{fig:cotangFan}
\end{figure} 

Again, we consider the ample anti-canonical divisor $K_{\bP(\Omega_{\bP^2})}$ and use $K_{\bP^1} = -2[0]$
as a representative of the canonical divisor on $\bP^1$. We obtain that
\[ -K_{\bP(\Omega_{\bP^2})} = 2D_{[0],(0,0)} + 2D_{[0],(0,1)} \,.\] 
The weight polytope of the corresponding support function is pictured in Figure \ref{fig:boxes}(b).

\begin{figure}[h]
\vspace*{-2ex}
\centering
\subfigure[$\Sigma = \tail \fan$]{\tailFan}\hspace*{8ex}
\subfigure[$\deg \fan$]{\degreeFan}
\caption{Tailfan and degree of $\fan\big(\bP(\Omega_{\bP^2})\big)$, cf.\ Example \ref{ex:cotang}.}
\label{fig:cotangPlus}
\end{figure}

\end{example}

\section{Okounkov Bodies}
\label{sec:OB}

The aim of this section is to compute Okounkov bodies of complexity-one $T$-varieties. After recalling
the general construction and some results from toric geo\-metry, we construct two types of admissible flags
in the complexity-one setting and describe the associated Okounkov bodies. Furthermore, we compute them for various examples.

If not stated otherwise all divisors in this section are supposed to be Cartier.

\subsection{Construction}
\label{subsec:OC}
In a series of papers \cite{okounkov1,okounkov2} on log-concavity of multiplicities Andrei Okounkov gave a procedure to associate
a convex set to a linear system on a projective variety. Although Okounkov essentially worked in the setting of ample line bundles,
the construction works perfectly well for big divisor classes. Robert Lazarsfeld and Mircea Musta\c{t}\u{a} thoroughly studied this
setting and recovered many fundamental results from the asymptotic theory of linear series, cf.\ \cite{okounkovBody}.

Let us now briefly recall the construction of the so-called Okounkov body as presented in \cite[Section 1]{okounkovBody}. Denote by
$X$ a projective variety of dimension $d$ and fix a flag 
\[Y_\bullet \;: \; X = Y_0 \supset Y_1 \supset Y_2 \supset \dots \supset Y_{d-1} \supset Y_d = \{\textnormal{pt}\}\,,\] 
consisting of subvarieties $Y_i$ of codimension $i$ in $X$ each of which is non-singular at the point $Y_d$. A flag $Y_\bullet$ as
above will be called an \emph{admissible flag}.

For any divisor $D$ on $X$ one can define a valuation-like function 
\[\nu_{Y_\bullet,D} \;: (H^0(X,\cO_X(D)) \setminus \{0\}) \to \bZ^d, \; s \mapsto \nu_{Y_\bullet,D}(s) = (\nu_1(s),\dots,\nu_d(s))\]
by an inductive procedure. Restricting to a suitable open neighborhood of the smooth point $Y_d$, we may assume that $Y_{i+1}$ is a
Cartier divisor on $Y_i$ for $0 \leq i \leq d-1$. 

To begin with, we set $\nu_1(s) = \ord_{Y_1}(s)$ where $\ord_{Y_1}(s)$ denotes the vanishing order of $s$ along $Y_1$. In other words,
it is equal to $\ord_{Y_1}(\divisor(s) + D)$. By choosing a local equation for $Y_1$ in $X$, our section $s$ determines in a natural
way a section $\wt{s}_1 \in H^0(X,\cO_X(D- \nu_1(s) Y_1))$ which does not vanish identically along $Y_1$. Restricting $\wt{s}_1$ to
$Y_1$ gives us a non-zero section $s_1 \in H^0(Y_1,\cO_{Y_1}(D-\nu_1(s) Y_1)$ and we set $\nu_2(s) = \ord_{Y_2}(s_1)$. We can define
the remaining $\nu_i(s)$ analogously and thus obtain the valuation vector $(\nu_1(s),\dots,\nu_d(s))$ associated to $s \in
\Gamma(X,\cO_X(D))$.

Observe that the valuation-like function $\nu_{Y_\bullet,\cdot}$ has the following properties:

\begin{enumerate}
\item Ordering $\bZ^d$ lexicographically,
\[\val{D}(s_1 + s_2) \geq \min \{\val{D}(s_1),\val{D}(s_2)\}\]
for any $s_1,s_2 \in \Gamma(X,\cO_X(D)) \setminus \{0\}$.
\item  For $s \in \Gamma(X,\cO_X(D)) \setminus \{0\}$ and $t \in \Gamma(X,\cO_X(E)) \setminus \{0\}$
\[\val{D+E}(s \otimes t) = \val{D}(s) + \val{E}(t)\,.\]
\end{enumerate}

\noindent
Working with a fixed divisor $D$ we will often simply write $\nu_{Y_{\bullet}}(s)$ instead of $\val{D}(s)$. Moreover, we denote by
$\nu_{Y_{\bullet}}(D)$ the set of all $\nu_{Y_{\bullet}}(s)$ for $s \in \Gamma(X,\cO_X(D)) \setminus \{0\}$.

\begin{definition}
\label{def:gradedSemigroup}
Let $X$ be a projective variety, $D$ a divisor on $X$ and $Y_\bullet$ a fixed admissible flag. The \emph{graded semigroup} of $D$ with
respect to the flag $Y_\bullet$ is the subsemigroup
\[\Gamma_{Y_\bullet}(D) = \big\{(\nu_{Y_\bullet}(s),m) \,\big|\, s \in \Gamma(X,\cO_X(mD)) \setminus \{0\},\, m \geq 0 \big\} \subset
\bN^d \times \bN\,.\]
\end{definition}

\begin{definition}
\label{def:OB}
Let $X$ be a projective variety, $D$ a divisor on $X$, and $Y_\bullet$ a fixed admissible flag. The \emph{Okounkov body} of $D$ with
respect to the flag $Y_\bullet$ is defined as 
\[\OB(D) = \ovl{\conv( \bigcup_{m\geq 1} 1/m \cdot \nu_{Y_{\bullet}}(mD))} \subset \bR^d\,.\] 
\end{definition}

By construction, we have that $\OB(D) \subset \posOrthant{\bR^d}$. It is shown in \cite[Theorem 2.3]{okounkovBody} that
$\vol_{\bR^d}(\OB(D)) =  \vol_X(D)/d!$ for a big divisor $D$ on a projective variety $X$ of dimension $d$, where 
\[\vol_X(D) = \limsup_{m \to \infty} \frac{\dim \Gamma(X,\cO_X(mD))}{m^d/d!}.\]
If $D$ is nef this quantity is equal to the top self-intersection number $D^d$, see \cite[p.\,148]{lazarsfeld1}. In particular,
$\OB(D)$ has a non-empty interior. Furthermore, $\OB(D)$ only depends on the numerical equivalence class of $D$
(cf.\ \cite[Proposition 4.1]{okounkovBody}), and $\OB(kD) = k \cdot \OB(D)$. This equality moreover says that $\OB(\xi)$ is well
defined for big classes $\xi \in N^1(X)_\bQ$. Indeed, we simply set
\[\OB(\xi) := \frac{1}{k}\OB(k \cdot \xi)\]
for some $k \in \bZ_{\geq 1}$ such that $k \cdot \xi \in N^1(X)$.

Using the correspondence between Cartier divisors and line bundles on $X$, we will sometimes switch notation from $\OB(D)$ to
$\OB(\cL)$ if $\cL \cong \cO_X(D)$.

Note that Okounkov bodies may very well be non-polyhedral, and even when polyhedral they often are not rational,
cf.\ \cite[6.2-6.3]{okounkovBody}. Nonetheless, a very nice feature is that the set of Okounkov bodies $\OB(\xi)$
for all big rational classes $\xi \in N^1(X)_\bQ$ fit together to a global convex object.

\begin{definition}
\label{def:globalOB}
Cf. \cite[Theorem 4.5]{okounkovBody}. Let $X$ be a projective variety and $Y_\bullet$ a fixed admissible flag. The
\emph{global Okounkov body} $\OB(X)$ of $X$ with respect to the flag $Y_\bullet$ is defined as the closed convex cone $\OB(X)
\subset \bR^d \times N^1(X)_\bR$ such that the fiber of the projection $\bR^d \times N^1(X)_\bR \to N^1(X)_\bR$ over any big
class $\xi \in N^1(X)_\bQ$ is equal to $\OB(\xi)$.
\end{definition}

By construction the global Okounkov body projects to the pseudoeffective cone 
\[\psEff = \ovl{\BigC(X)} \subset N^1(X)_\bR\] 
which is the closure of the big cone $\BigC(X)$.

\subsection{Toric Varieties}
\label{subsec:toricOB}
Let $\TV(\Sigma)$ be a smooth projective toric variety of dimension $d$ which is given by a fan $\Sigma$ in $N_\bQ$, and let
$m$ be the number of rays $\rho \in \Sigma(1)$ corresponding to the torus invariant prime divisors in $X$. Recall the exact
sequence 
\[0 \longrightarrow M \stackrel{\iota}{\longrightarrow} \bZ^m \stackrel{\textnormal{pr}}{\longrightarrow} \pic(\TV(\Sigma))
\longrightarrow 0\,. \] 
Supposing the admissible flag $Y_\bullet$ to be invariant, one can order the invariant prime divisors of $\TV(\Sigma)$ in such
a way that $Y_i = D_1 \cap\,\dots\,\cap D_i$. The set of the corresponding rays $\{\rho_1, \dots \rho_d\}$ clearly spans a
smooth $d$-dimensional cone which we denote by $\sigma$. It corresponds to the fixed point $Y_d$. Taking the primitive generators
$n_i$ of these rays as a basis for the lattice $N$, we obtain a splitting of the above exact sequence into
\[\psi:\; \bZ^d \times \pic(\TV(\Sigma)) \longrightarrow \bZ^m\] 
with $\psi^{-1}(D) = (q(D),\textnormal{pr}(D))$ and $q:\; \bZ^m \to \bZ^d$ being the projection onto the first $d$ coordinates.
We denote by $\phi: M \to \bZ^d$ the map which is given by $\phi(u) = (\langle u, n_i \rangle)_{1\leq i \leq d}$. Moreover, we
denote by $P_D \subset M_\bR$ the polytope whose lattice points correspond to the homogeneous global sections of $\cO(D)$.

\begin{proposition}
\label{prop:obToricVar}
Cf.\ \cite[Proposition 6.1]{okounkovBody}. Let $\TV(\Sigma)$ be a smooth projective toric variety, and let $Y_\bullet$ be an
admissible flag of invariant subvarieties chosen as above. 
\begin{enumerate}
\item Given any big equivariant line bundle $\cL$ on $\TV(\Sigma)$, let $D$ be the unique $T$-invariant divisor such that
$\cL \simeq \cO(D)$ and its restriction to the affine chart $U_\sigma$ is trivial. Then we have that 
\[\OB(\cL) = \phi_{\bR}(P_D)\,.\]
\item The global Okounkov body $\OB(\TV(\Sigma))$ is the inverse image of the non-negative orthant $\bR^m_{\geq 0} \subset \bR^m$
under the isomorphism
\[\psi_\bR: \; \bR^d \times \pic(\TV(\Sigma))_\bR \stackrel{\cong}{\longrightarrow} \bR^m\,. \]
\end{enumerate}
\end{proposition}

\subsection{Two Types of Invariant Flags for Complexity-One $T$-Varieties}
\label{subsec:flags}
Let $\TV(\fan)$ be a projective $T$-variety of complexity one which contains at least one smooth point $\xfix$ that is fixed
under the torus action. The aim of this section is to construct $T$-invariant admissible flags $Y_\bullet$ in $\TV(\fan)$
with $Y_d = \xfix$ which will then be used for the computation of Okounkov bodies.

As before, we denote by $\cP \subset C$ a non-empty finite set of points in $C$ such that the slice $S_Q$ over a point
$Q \in C \setminus \cP$ is trivial.

\begin{definition}
\label{def:generalPoint}
A point $Q \in C \setminus \cP$ is called \emph{general}.
\end{definition}

A slice $\fan_Q$ for a general point $Q \in C$ is equal to $\Sigma := \tail \fan$, meaning that the fiber of the
quotient map $\pi: \wt{\TV}(\fan) \to C$ is equal to the toric variety $\TV(\Sigma)$. In particular it is reduced and irreducible
(cf.\ \cite[Section 7]{tvar1}).

In the following, we will present the construction of several types of admissible $T$-invariant flags in $\TV(\fan)$ which will
depend upon the choice of a smooth fixed point $\xfix \in \TV(\fan)$.

\subsubsection*{General Flags}
\label{subsub:generalFlags}
\ \\[.5ex]
$\mathbf{G_1}$ \; We assume that $\xfix$ lies over a general point $Q \in Y$ and that the maximal cone $\sigmafix \in \Sigma$
corresponding to $\xfix$ is not marked, i.e.\ $\sigmafix \notin \cC(\fan)$ (see (\ref{subsec:marking})). We
set $Y_1 := r(\pi^{-1}(Q)) \cong \TV(\Sigma)$ and proceed as in the toric case for the remaining elements of the flag. Namely, we
label the rays in the smooth cone $\sigmafix$ from $1$ to $d-1$, i.e.\ $\sigmafix(1) = \{\rho_1,\dots,\rho_{d-1}\}$ and we define
$\Delta_F(k) := \bangle{\rho_1,\dots,\rho_k}$. Thus, we see that $\Delta_F(k)$ corresponds to a $T$-orbit of codimension $k$ in
$Y_1$ which allows us to define $Y_{i+1} := r(\ovl{\orb(\Delta_F(i))}) \subset Y_1$ for $1 \leq i \leq d-1$. It is not hard to see
that such a flag is admissible. \\[1ex]
$\mathbf{G_2}$ \; We assume that the maximal cone $\sigmafix \in \Sigma$ corresponding to $\xfix$ is marked, i.e.\
$\sigmafix \in \cC(\fan)$ and smooth. Hence, we can now proceed as in the construction of an admissible flag of type $\mathbf{G_1}$
by picking a general point $Q$ and identifying $\xfix$ with the orbit that corresponds to the cone $\sigmafix \in \tail
\fan = \fan_Q$.

\subsubsection*{Toric Flags}
\label{subsubsec:toricFlags}
\ \\[.5ex]
$\mathbf{T_1}$ \; We assume that the maximal cone $\sigmafix \in \Sigma$ corresponding to $\xfix$ is not marked and that $\xfix$ 
lies over a point $P \in \cP$. Since $\xfix$ is smooth we are in a formal-locally toric situation. Indeed, after a
suitable refinement of the invariant covering, we can assume that $\xfix$ is contained in an affine open subset $\TV(\cD^{\sigmafix})
\subset \TV(\fan)$ for a p-divisor $\cD^{\sigmafix}$ with locus contained in $(Y \setminus \cP) \cup \{P\}$ and tailfan $\sigmafix$.
According to \cite[Theorem 3.3]{candiv}, we have that $\big(\TV(\cD^{\sigmafix}),\xfix \big)$ is formally isomorphic to the smooth
affine toric variety $\big(\TV(\deltafix),\orb(\deltafix)\big)$ with
\[\deltafix = \ovl{ \bQ_{\geq 0} \cdot \big(\{1\} \times \cD^{\sigmafix}_{P})} \subset \bQ_{\geq 0} \times N_\bQ\,.\]
The rays of $\deltafix$ are given through the vertices of $\cD^{\sigmafix}_P$ in height $1$ and the rays of the tailcone $\sigmafix$
in height $0$. The admissible flag then arises as in the toric setting by an enumeration of the rays of $\deltafix$, cf.\
(\ref{subsec:toricOB}).\\[1ex]
$\mathbf{T_2}$ \; We assume that the maximal cone $\sigmafix \in \Sigma$ corresponding to $\xfix$ is marked. Since $\xfix$ is smooth
we are in a Zariski-locally toric situation. Indeed, according to \cite[Proposition 3.1]{candiv} we have an affine open $T$-invariant
subset $\TV(\cD^{\sigmafix}) \subset \TV(\fan)$ such that
\[\cD^{\sigmafix} \cong \cD^{\sigmafix}_{P_1} \otimes [P_1] + \cD_{P_2}^{\sigmafix} \otimes [P_2]\,.\]
Thus, after adding a principal p-divisor, we may assume that $\cD^{\sigmafix} \in \fan$ is equal to the r.h.s. Hence, we see that
$\TV(\cD^{\sigmafix})$ is isomorphic to the smooth affine toric variety $\TV(\deltafix)$ with
\[\deltafix = \ovl{ \bQ_{\geq 0} \cdot \big(\{1\} \times \cD^{\sigmafix}_{P_1} \; \cup \; \{-1\} \times \cD^{\sigmafix}_{P_2}\big)}
\subset \bQ \times N_\bQ\,.\]
The rays of $\deltafix$ are given through the vertices of $\cD^{\sigmafix}_{P_1}$ and $\cD^{\sigmafix}_{P_2}$. In particular, we
have a natural upgrade of the torus action. We now construct an admissible flag as in the toric setting by numbering the rays of
$\deltafix$.

\begin{remark}
a) If the maximal cone $\sigmafix$ corresponding to the smooth fixed point $\xfix$ is marked then we must have that $C=\bP^1$
(see \cite[Proposition 3.1]{candiv}).\\
b) Considering a toric variety $\TV(\Sigma)$ as a complexity-one $T$-variety via the downgrade method we presented in
(\ref{subsec:toricDown}), it is not hard to check that the admissible invariant flags constructed above comprise those which are
invariant under the original (big) torus action.\\
c) Note that the existence of one admissible general flag implies the existence of a one parameter family of these since
$Q$ may be chosen from $\bP^1 \setminus \cP$.
\end{remark}

The subsequent lemma by Jos\'{e} Gonz\'{a}lez (cf.\ \cite[Lemma 4.5]{obToricVB}) is an important ingredient for the computation of
Okounkov bodies in $T$-invariant settings.

\begin{lemma}
Let $X$ be an affine $T$-variety together with an admissible flag $Y_\bullet \,: X = Y_0 \supset \dots \supset Y_d$  of normal
$T$-invariant subvarieties such that $Y_{i+1} = \divisor h_{u_i} \subset Y_i$ for a rational semi-invariant function $h_{u_i}$,
$(1 \leq i \leq d-1)$, $u_i \in M$. For a rational function $g \in \bK(X)$ that decomposes as $g = \sum g_{u_j}$ into homogeneous
components with respect to elements $u_j \in M$, we have that $\nu_{Y_\bullet}(g) \in \{\nu_{Y_\bullet}(g_{u_j})\}$.
\end{lemma}

It is indeed not hard to check that our construction of general and toric flags allows for a reduction to the setting of the above
lemma.

\begin{remark}
\label{rmk:flagsGonzalez}
In \cite[Section 4.1]{obToricVB}, Jos\'{e} Gonz\'{a}lez gave a construction for a $T$-invariant flag on projectivized rank two toric
vector bundles over smooth projective toric varieties. Using the description of these projectivized bundles in terms of p-divisors
(cf.\ \cite[Proposition 8.4]{tvar2}), one sees that those flags are of type $\mathbf{T_2}$.
\end{remark}

\subsection{Okounkov Bodies for General Flags}
\label{subsec:computationOBgeneral}

Let us recall some notions which were introduced in (\ref{subsec:globalSections}). Given a $T$-invariant Cartier divisor $D_h$
on $\TV(\fan)$ we denote by $D_{\ul{h}}$ the Cartier divisor which is defined on $\TV(\tail \fan)$ via the linear part of $h$.
Furthermore, we have the map
\[\hstar_P: \Box_h \to \bQ\,, \quad u \mapsto \minvert (u- h_P)\,,\]
for every point $P \in Y$. For the ease of later computations, we introduce the following notion. A $T$-invariant divisor $D_h$
on $\TV(\fan)$ is called \emph{normalized} with respect to the general flag $Y_\bullet$ if $h_Q|_{\cD_Q^{\sigmafix}} \equiv 0$.
In particular, this implies that $\hstar_Q \equiv 0$.

\begin{theorem}
\label{thm:localOBgeneral}
Let $\TV(\fan)$ be a projective $T$-variety of complexity one together with a general flag $Y_\bullet$. Consider a $T$-invariant
big divisor $D_h$ on $\TV(\fan)$ which is normalized with respect to $Y_\bullet$. Denote by $D_{\ul{h}}$ the associated invariant
divisor on the toric variety $Y_1 = \TV(\Sigma)$, where $\Sigma = \tail \fan$, and consider the induced flag $Y_{\geq 1}$ on $Y_1$. 
Then we have that
\[\OB(D_h) = \Big\{(x,w) \in \bR \times \bR^{d-1} \, \big| \; w \in \Delta_{Y_{\geq 1}}(D_{\ul{h}}), \; 0 \leq x \leq \degree
\hstar(\phi_{\bR}^{-1}(w))\big\},\]
where $\phi_{\bR}$ is equal to the map which was introduced in the toric setting of (\ref{subsec:toricOB}). Moreover,
$\Delta_{Y_{\geq 1}}(D_{\ul{h}}) = \phi_\bR(\Box_{\ul{h}})$ denotes the Okounkov body of $D_{\ul{h}}$ on $Y_1$ with respect to the
flag $Y_{\geq 1}$. In particular, $\OB(D_h)$ is a rational polytope.
\end{theorem}

\begin{proof}
Note that 
\[m \, \hstar_P \big(\frac{1}{m}u \big) = (mh)_P^*(u)\, \textnormal{ for } m \in \bZ_{\geq 1}\,,\; u \in \Box_h\,.\]
Let us first prove the inclusion ``\,$\subset$\,''. It is enough to show that $\frac{1}{m}\nu(s)$ is an element of the r.h.s.\
for any homogeneous non-zero section $s \in \Gamma(\TV(\fan),\cO(D_{mh}))$ and any $m \geq 1$. We write $s = f\chi^u$ where 
$u = \phi^{-1}_{\bR}(w)$ denotes the weight of $s$ and $f \in \bK(\bP^1)$. Due to convexity and the results in the toric setting
for $Y_{i \geq 1}$, it is enough to show that 
\[\degree \hstar\big(\frac{1}{m}u \big) \geq \frac{1}{m}\nu_1(s) \geq 0\,.\]
Since $D_h$ is normalized we have that $\hstar_Q \equiv 0$ and $\nu_1(s) = \ord_Q(f) \geq 0$. Furthermore, $\nu_1(s)$ is bounded
above by $\sum_{P \in C}\floor{(mh)_P^*(u)}$. Thus, we arrive at
\[m \degree \hstar \big(\frac{1}{m}u \big) = \degree \,(mh)^*(u) \geq \sum_{P \in C} \floor{(mh)^*_P (u)} \geq \nu_1(s) \geq 0 \,.\]

For the other inclusion, consider a point $(x,w) \in \bQ \times \bQ^{d-1}$ of the r.h.s., i.e.\
\[\Delta_{Y_{\geq 1}}(D_{\ul{h}}) \ni w = \phi_{\bQ}(u)\] 
for some $u \in \Box_h$. Due to convexity and the fact that $x \leq \degree \hstar(u)$ with $u = \phi_\bQ^{-1}(w)$ it is enough to
show that $(\degree \hstar(u),w) \in$ l.h.s.

\noindent
Since we only have finitely many non-trivial slices, each of which is a finite subdivision of $N_\bQ$, there exists a natural number
$N$ such that $\Box_{Nh} \ni Nu \in M$, and $(Nh)_P^*(Nu)$ is an integer for every $P \in \cP$. So the round-down is no longer necessary
and we have
\[N \degree \hstar(u) = N \sum_{P \in C}h^*_P(u) = \sum_{P \in C} (Nh)^*_P \big( Nu \big) = 
\sum_{P \in C} \floor{(Nh)^*_P \big( Nu \big)} \geq 0. \]
But then we can find a sequence of positive integral multiples $k_iN$ of $N$ and homogeneous sections
$s_i \in \Gamma(\TV(\fan),\cO(D_{k_iNh}))$ of weight $k_iNu$ such that
\[\limsup_{i \to \infty}\frac{\nu_1(s_i)}{k_iN} = \deg \hstar(u)\]
which completes the proof.
\end{proof}

Thus, Theorem \ref{thm:localOBgeneral} also relates divisorial polytopes $(\Psi,\Box,C)$ to Okounkov bodies $\OB(D_{\Psi^*})$,
cf.\ (\ref{subsec:divisorialPolytopes}).

\begin{corollary}
\label{cor:localOBgeneralDivPol}
Fixing a general flag $Y_\bullet$ in the polarized complexity-one $T$-variety $\big(\TV(\Xi(\Psi)),\cO(D_{\Psi^*})\big)$ related to the
divisorial polytope $(\Psi,\Box,C)$, the associated Okounkov body $\OB(D_{\Psi^*})$ arises, up to translation, from the convex hull
of the graph of the function $\sum_{P \in C} \Psi_P$ over the polytope $\Box$.
\end{corollary}

\begin{remark}
a) Fixing a big $T$-invariant divisor $D_h$ together with an enumeration of the rays of $\sigmafix$, Theorem \ref{thm:localOBgeneral}
also shows that all resulting Okounkov bodies are identical since there is no dependence on $Q \in C \setminus \cP$. \\
b) Recall that the correspondence between divisorial polytopes and polarized projective complexity-one $T$-varieties generalizes
the correspondence between lattice polytopes and polarized projective toric varieties. In the same vein Corollary
\ref{cor:localOBgeneralDivPol} generalizes Proposition \ref{prop:obToricVar}. 
\end{remark}

\subsection{Okounkov Bodies for Toric Flags}
\label{subsec:computationOBtoric}

Let $\TV(\fan)$ be a projective $T$-variety of complexity one together with a fixed toric flag $Y_\bullet$. A $T$-invariant
divisor $D_h$ is called \emph{normalized} with respect to $Y_\bullet$ if $h|_{\cD^{\sigmafix}} \equiv 0$. 

Moreover, we introduce the map 
\[c_{Y_\bullet}: \tcadiv \big(\TV(\fan) \big)_\bQ \to \bQ^d, \quad c_{Y_\bullet}(D_h)_i = \coeff_{D_i} D_h \,, \]
where $D_i$ is the Weil divisor in $\TV(\fan)$ which is associated to the i'th ray of the cone $\delta_{\textnormal{fix}}$
arising from $Y_\bullet$.

\subsubsection{Using a Toric Flag of Type $\mathbf{T_1}$}
\label{subsubsec:OBtoricFlagATwo}
Recall that $\xfix$ lies over a point $P \in \cP$ and the associated cone $\sigmafix$ is not marked. Embedding $\cD^{\sigmafix}_P$
into $\{1\} \times N_{\bQ}$, we obtain a smooth cone $\deltafix \subset \bQ \times N_{\bQ}$. Since $D_h$ is normalized, a global
section $f\chi^u \in \Gamma(\TV(\fan),\cO(D_h))_u$ turns into the rational function of weight $[\ord_P f,u] \in \bZ \times M$ where
$\bZ \times M$ is the character lattice of the big torus acting upon $\TV(\deltafix)$. In addition, we define
\[\phi: \bZ \times M \to \bZ^d, \quad \phi\big([\ord_P f, u]\big) = \Big(\big\langle [\ord_P f,u],n_i
\big\rangle\Big)_{1 \leq i \leq d}\,,\]
where $n_i$ is the primitive generator of the $i$-th ray of $\deltafix$ fixed by the flag $Y_\bullet$. Note the similarity to the map
used in (\ref{subsec:toricOB}). The next statement now follows easily from the toric discussion and the definition of
$\nu_{Y_\bullet,D_h}$.

\begin{lemma}
\label{lem:valToricFlagATwo}
Let $\TV(\fan)$ and $Y_\bullet$ be as above. For a $T$-invariant divisor $D_h$ on $\TV(\fan)$ we have
\[\nu_{Y_\bullet,D_h}(f\chi^u) = \phi\big([\ord_P f,u]\big) + c_{Y_\bullet}(D_h)\,,\]
where the last summand vanishes if $D_h$ is normalized.
\end{lemma}

\begin{proposition}
\label{prop:localOBtoricATwo}
Let $\TV(\fan)$ be a projective $T$-variety of complexity one together with a toric flag $Y_{\bullet}$ of type $\mathbf{T_1}$
and a normalized big $T$-invariant divisor $D_h$. The Okounkov body $\OB(D_h)$ then results from a translation of the rational polytope
\[W(h) \ldef \ovl{\Big \{(x,u) \in \bQ \times \Box_{\ul{h}} \;\big|\; 0 \leq x + \hstar_P(u) \leq \deg \hstar(u) \Big\}} \subset \bR
\times \bR^{d-1}.\]
which is induced by the ordered set of primitive generators of the rays of $\deltafix$, i.e.\ an element $w \in W(h)$ gives us
\[\big(\bangle{w,n_1},\dots,\bangle{w,n_d}\big) \in \bR^d.\]
In particular, $\OB(D_h)$ is again a rational polytope.
\end{proposition}

\begin{proof}
The proof is essentially analogous to the proof of Theorem \ref{thm:localOBgeneral}.
\end{proof}

\subsubsection{Using a Toric Flag of Type $\mathbf{T_2}$}
\label{subsubsec:OBtoricFlagBTwo}
For the computation of the Okounkov body with respect to an admissible flag of type $\mathbf{T_2}$ we assume that
\[\cD^{\sigmafix} = \cD^{\sigmafix}_{P_1} \otimes [P_1] + \cD^{\sigmafix}_{P_2} \otimes [P_2]\,.\]
Embedding $\cD^{\sigmafix}_{P_1}$ in $\{1\} \times N_{\bQ}$ and $\cD^{\sigmafix}_{P_2}$ in $\{-1\} \times N_{\bQ}$, we obtain a
smooth cone $\deltafix \subset \bQ \times N_{\bQ}$. As $D_h$ is normalized, a global section $f\chi^u \in \Gamma(\TV(\fan),\cO(D_h))_u$
turns into the rational function of weight $[\ord_{P_1}f,u] =[-\ord_{P_2}f,u] \in \bZ \times M$. Again, the latter lattice is the
character lattice of the big torus that acts upon $\TV(\deltafix)$. With the very same notation as in the previous section, we define
\[\phi: \bZ \times M \to \bZ^d, \quad \phi\big([\ord_{P_1}f, u]\big) = \Big(\big\langle [\ord_{P_1}f,u],n_i
\big\rangle\Big)_{1 \leq i \leq d}\,.\]
This gives us

\begin{lemma}
\label{lem:valToricFlagBTwo}
Let $\TV(\fan)$ and $Y_\bullet$ be as above. For a $T$-invariant divisor $D_h$ on $\TV(\fan)$ we have that
\[\nu_{Y_\bullet,D_h}(f\chi^u) = \phi\big([\ord_{P_1}f,u]\big) + c_{Y_\bullet}(D_h)\,,\]
where the last summand vanishes if $D_h$ is normalized.
\end{lemma}

\begin{proposition}
\label{prop:localOBtoricBTwo}
Let $\TV(\fan)$ be a projective $T$-variety of complexity one together with a toric flag $Y_{\bullet}$ of type $\mathbf{T_2}$
and a normalized big $T$-invariant divisor $D_h$. The Okounkov body $\OB(D_h)$ then results from a translation of the rational polytope
\[W(h) \ldef \ovl{\Big \{(x,u) \in \bQ \times \Box_{\ul{h}} \;\big|\; 0 \leq x + \hstar_{P_1}(u) \leq \deg \hstar(u) \Big\}} \subset \bR \times \bR^{d-1}.\]
which is induced by the ordered set of primitive generators of the rays of $\deltafix$, i.e.\ an element $w \in W(h)$ gives us
\[\big(\bangle{w,n_1},\dots,\bangle{w,n_d}\big) \in \bR^d.\]
Hence, $\OB(D_h)$ is also a rational polytope.
\end{proposition}

\begin{proof}
Replacing $P$ by $P_1$ the proof is identical to the proof of Proposition \ref{prop:localOBtoricATwo}.
\end{proof}

\subsection{Examples}
\label{subsec:exLocalOB}

\subsubsection{Revisiting Toric Geometry}
\label{subsec:newOBtv}
Considering a toric variety $\TV(\Sigma)$ with big torus $T$, we may downgrade to a torus action of complexity one where we denote
the smaller torus by $T'$. It turns out that the set of $T'$-invariant admissible flags we have described in the previous section is
much bigger than the set of admissible flags which are invariant under the action of the big torus $T$. Indeed, we essentially have
a one-parameter family of choices for a general flag (depending on the choice of the point
$Q \in \bP^1 \setminus \cP$) whereas, in the toric setting, we are restricted to the flags which are associated to the finite
(and possibly empty) set of $T$-fixed points, cf.\ (\ref{subsec:toricOB}). This means, for example, that we will now be able to
compute Okounkov bodies even if there is no smooth $T$-invariant fixed point at all (see Example \ref{ex:nonSmoothTV}). Moreover,
computations of Okounkov bodies for a line bundle $\cO(D)$ with respect to some general $T'$-invariant flags can yield convex bodies
that differ considerably from $P_D$.

In the following, we will illustrate a few new features of Okounkov bodies associated to ample line bundles on some genuinely toric
$\bK^*$-surfaces.

\begin{figure}[b]
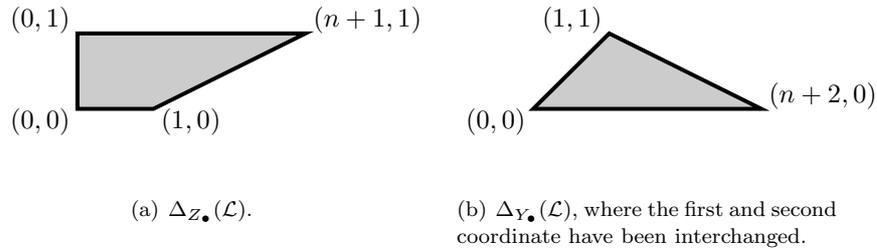

\centering
\subfigure[$\Delta_{Z_\bullet}(\cL).$]{\hirzebruchNPolytope}
\hspace*{5ex}
\subfigure[$\Delta_{Y_\bullet}(\cL)$, where the first and second coordinate have been interchanged.]{\hirzebruchNOB}
\caption{Okounkov bodies associated to different flags for an ample line bundle $\cL$ on $\bF_n$, cf.\ Examples \ref{ex:hirzebruchOB}
and \ref{ex:hirzebruchDegenerationOB}.}
\label{fig:hirzebruchNPolytopes}
\end{figure}

\begin{example}
\label{ex:hirzebruchOB}
We return to our Example \ref{ex:hirzebruch} and consider the $n$'th Hirzebruch surface $\bF_n$ as a $\bK^*$-surface together with the
ample line bundle $\cL = \cO(D_{\rho_2} + D_{\rho_3})$. Our aim is to perform computations with respect to all flags discussed so
far.\\[1ex]
\noindent
\emph{The Toric Setting}. The toric Okounkov body $\Delta_{Z_\bullet}(\cL)$ with respect to the flag
\[Z_\bullet: \quad \bF_n \supset D_{\rho_0} \supset (D_{\rho_0} \cap D_{\rho_1})\]
then is given by the polytope which is pictured in Figure \ref{fig:hirzebruchNPolytopes}(a).\\[2ex]
\noindent
\emph{A General Flag of Type $\mathbf{G_1}$.} To give such a flag, we choose a parabolic fixed point $\xfix$ represented by the
interval $[0 \; \infty)$ in the slice $\fan_Q$ for $Q \in \bP^1 \setminus\{0,\infty\}$. Then we define
\[Z^1_\bullet: \quad \bF_n \supset r(\pi^{-1}(Q)) \supset (Q,[0 \; \infty))\,.\]
Note that $D_h = D_{\rho_2} + D_{\rho_3}$ already is normalized with respect to $Z_{\bullet}^1$. Using Theorem \ref{thm:localOBgeneral},
an easy calculation then shows that
\[\Delta_{Z^1_{\bullet}}(\cL) = \conv\{(0,0),(1,0),(1,1),(0,n+1)\}.\]
\noindent
\emph{A General Flag of Type $\mathbf{G_2}$.} We consider a general point $Q \in \bP^1 \setminus \{0,\infty\}$ together with the flag
\[Z^2_\bullet: \quad \bF_n \supset r(\pi^{-1}(Q)) \supset (Q,(-\infty \; 0])\,.\]
Note that $(Q,(-\infty\; 0])$ corresponds to the elliptic fixed point. We take $D_h = (n+1)D_{\rho_0} + D_{\rho_1}$ which is linear
equivalent to $D_{\rho_2} + D_{\rho_3}$ and normalized with respect to $Z^2_{\bullet}$. An easy computation then shows that
\[\Delta_{Z^2_{\bullet}}(D_h) = \conv\{(0,0),(0,n+1),(1,n),(1,n+1)\}.\]

\begin{figure}[t]
\centering
\subfigure[$\hstar_0$]{\hirzHstarZero}\hspace*{10ex}
\subfigure[$\hstar_\infty$]{\hirzHstarInfty}
\caption{Graphs of $\hstar_0$ and $\hstar_\infty$, cf.\ Example \ref{ex:hirzebruchOB}.}
\label{fig:hirzHstar}
\end{figure}

\begin{figure}[b]
\centering
\subfigure[$\hstar_0$]{\hirzebruchHstarZero}\hspace*{10ex}
\subfigure[$\hstar_\infty$]{\hirzebruchHstarInfty}
\caption{Graphs of $\hstar_0$ and $\hstar_\infty$, cf.\ Example \ref{ex:hirzebruchDegenerationOB}.}
\label{fig:hirzebruchHstar}
\end{figure}

\noindent
\emph{A Toric Flag of Type $\mathbf{T_1}$.} Let us consider an admissible flag of type $\mathbf{T_1}$ which is associated to the
hyperbolic fixed point $\xfix$ represented by the interval $[-1/n \quad 0]$ in the slice $\fan_0$. Numbering the rays of $\deltafix$
by $(r_1,r_2):= \big(\bQ_{\geq 0}(1,0),\bQ_{\geq 0}(n,-1)\big)$, we set
\[Z^3_\bullet: \quad  \bF_n \supset D_{r_1} \supset (D_{r_1} \cap D_{r_2})\,.\]
Using linear equivalence, we pass from $D_{\rho_2} + D_{\rho_3}$ to $D_h := D_{\rho_0} + D_{\rho_3}$ to obtain a divisor which is
normalized with respect to $Z^3_\bullet$. The graphs of the functions $\hstar_0$ and $\hstar_\infty$ are given in Figure
\ref{fig:hirzHstar}. Furthermore, we compute that 
\[W(h) = \conv \{(0,-1),(0,0),(1,-1),(1,0),(1,n)\} \subset \bR^2\,.\] 
Translating $W(h)$ via $(n_{r_1},n_{r_2})$ gives us 
\[\Delta_{Z^3_{\bullet}}(D_h) = \conv \{(0,1),(0,0),(1,n+1),(1,0)\}.\]

\noindent
Moreover, it is not hard to check that flags of type $\mathbf{T_1}$ with respect to the parabolic fixed points over $0$ or $\infty$
yield identical Okounkov bodies.\\[1ex]
\noindent
\emph{A Toric Flag of Type $\mathbf{T_2}$.} Finally, we compute the Okounkov body for a flag of type $\mathbf{T_2}$ with respect to the
elliptic fixed point $\xfix$. First, we order the rays of the induced cone $\deltafix$ by $(r_1,r_2) := \big(\bQ_{\geq 0}(1,-1),
\bQ_{\geq 0}(-1,0)\big)$ and set
\[Z^4_{\bullet}: \quad \bF_n \supset D_{r_1} \supset (D_{r_1} \cap D_{r_2})\,.\]
Moreover, we take $D_h = (n+1)D_{\rho_0} + D_{\rho_1}$ which is linear equivalent to $D_{\rho_2} + D_{\rho_3}$ and normalized with
respect to $Z^4_{\bullet}$. With these data we compute 
\[W(h) = \conv\{(0,-n-1),(0,0),(-1,-n-1),(-1,-n)\}\]
and obtain
\[\Delta_{Z^4_{\bullet}}(D_h) = \conv\{(n+1,0),(0,0),(n,1),(n-1,1)\}.\]
\end{example}

\noindent
\emph{Concluding Remark}. This particular downgrade did not give us any ``new'' polytopes. Indeed, one easily checks that all of them
can be transformed into $\Delta_{Z_\bullet}(\cL)$ by an affine lattice isomorphism.

\begin{example}
\label{ex:hirzebruchDegenerationOB}
In contrast to Example \ref{ex:hirzebruch}, we now choose the subtorus action which arises from the following data:
\[F = \left(\begin{array}{c} 1 \\ 1 \end{array} \right)\,, \qquad  P = \left(\begin{array}{cc} 1 & -1 \end{array}\right)\,,
\qquad  s = \left(\begin{array}{cc} 1 & 0 \end{array}\right)\,.\]
The associated divisorial fan is given in Figure \ref{fig:hirzebruchVar}. Furthermore, let $Q \in \bP^1 \setminus \{0,\infty\}$ be a
general point and fix the following general flag of type $\mathbf{G_1}$:
\[Y_\bullet: \quad \bF_n \supset r(\pi^{-1}(Q)) \supset (Q,[0 \;\; \infty))\,.\]
See Figure \ref{fig:hirzebruchHstar} for the graphs of $\hstar_P$ and the right hand polytope in Figure \ref{fig:hirzebruchNPolytopes}
for a picture of the resulting Okounkov body $\OB(\cL)$. Observe that this polytope corresponds to the toric variety $\bP(1,1,n+2)$. 
\end{example}

\begin{figure}[t]
\centering \hspace*{-8ex}
\subfigure[Relevant slices of $\fan$.]{\hirzebruchVarFan}
\subfigure[Tailfan and degree of $\fan$.]{\hirzebruchVarPlus}
\caption{Divisorial fan associated to $\bF_n$, cf.\ Example \ref{ex:hirzebruchDegenerationOB}.}
\label{fig:hirzebruchVar}
\end{figure}

\begin{figure}[b]
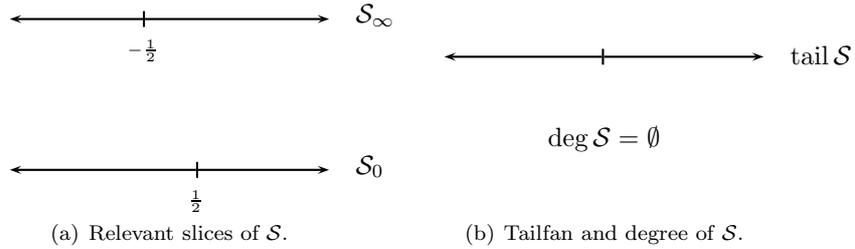

\centering \hspace*{-8ex}
\subfigure[Relevant slices of $\fan$.]{\toricSurfaceSimpleFan}
\subfigure[Tailfan and degree of $\fan$.]{\toricSurfaceSimplePlus}
\caption{Divisorial fan associated to $\TV(\Sigma)$, cf.\ Example \ref{ex:nonSmoothTV}.}
\label{fig:toricSurfaceSimple}
\end{figure}

\begin{example}
\label{ex:nonSmoothTV}
We consider the projective toric surface $\TV(\Sigma)$ whose primitive generators of the rays are given in the following list:
\[\begin{array}{cccc} \nu_1 = (1,1)\,, & \nu_2 = (-1,1)\,, & \nu_3 = (-1,-1)\,, & \nu_4 = (1,-1)\,.\end{array}\]
Since all of its fixed points are singular we cannot find an admissible flag as chosen in the toric setting of (\ref{subsec:toricOB}).
Nevertheless, we may perform a downgrade and choose the subtorus action that comes from 
\[F = \left(\begin{array}{c} 1 \\ -1 \end{array} \right)\,, \qquad  P = \left(\begin{array}{cc} 1 & 1 \end{array}\right)\,, \qquad
 s = \left(\begin{array}{cc} 1 & 0 \end{array}\right)\,.\]
The associated divisorial fan is given in Figure \ref{fig:toricSurfaceSimple}. Next, we consider the ample line bundle
$\cL = \cO(2D_{\rho_2} + 2D_{\rho_3})$ which is given by the following generators of the global sections over the affine charts
$U_{\sigma_i}$:
\[u_{\sigma_1} = [0 \;\; 0]\,,\quad u_{\sigma_2} = [1 \; -1]\,,\quad u_{\sigma_3} = [2 \;\; 0]\,,\quad u_{\sigma_4} = [1 \;\; 1]\,.\] 
Let $Q \in \bP^1 \setminus \{0,\infty\}$ and fix the following general flag of type $\mathbf{G_1}$:
\[Y_\bullet: \quad  \TV(\Sigma) \supset r(\pi^{-1}(Q)) \supset (Q,[0 \;\; \infty))\,.\]
Invoking Theorem \ref{thm:localOBgeneral}, we see that $\Delta_{Y_\bullet}(\cL)$ is equal to
\[\conv\{(0,0),(0,2),(1,0),(1,2)\}.\]
\end{example}

\subsubsection{The Anti-Canonical Bundle on $\bP(\Omega_{\bP^2})$}
\label{ex:cotangentOB}
\ \\
\emph{A Toric Flag.} For the construction of a toric flag of type $\mathbf{T_2}$, we consider the p-divisor $\cD_{\sigmafix}$
whose tailcone $\sigmafix$ is generated by the rays $(1,0)$ and $(1,1)$. Note that the polyhedral coefficient $\cD_{\infty}$ is
trivial. Hence, we have that $\TV(\cD_{\sigmafix}) = \TV(\deltafix) = \bA^3$ with $\deltafix$ being spanned by the rays
$\rho_1,\rho_2$ and $\rho_3$ whose primitive generators are $(1,0,0)$, $(1,0,1)$ and $(-1,1,0)$, respectively. We also take this
enumeration for the definition of our flag, i.e.\
\[Y_\bullet: \quad \TV(\deltafix) \; \supset \; D_{\rho_1} \supset \; (D_{\rho_1} \cap D_{\rho_2}) \; \supset \; (D_{\rho_1}
\cap D_{\rho_2} \cap D_{\rho_3}) = \xfix \,.\]
An easy calculation according to Proposition \ref{prop:localOBtoricBTwo} shows that $\OB(-K_X) \subset \bR^3$ is the convex
polytope whose vertices are represented as the columns of the following matrix
\[\left(\begin{array}{ccccccc} 0 & 2 & 0 & 2 & 2 & 0 & 0 \\ 0 & 0 & 2 & 2 & 0 & 2 & 0 \\ 0 & 0 & 0 & 0 & 2 & 2 & 4 \end{array}\right).\]
Interchanging the $y$ and $z$-axes and scaling with the factor $1/2$ gives us the same polytope which was computed in
\cite[Example 6.1]{obToricVB}.

\noindent
\emph{A General Flag.} Let $Q \in \bP^1 \setminus \{0,1,\infty\}$. We consider the maximal cone spanned by the rays
$\delta_1 = \bQ_{\geq 0 }\cdot (1,0)$ and $\delta_2 = \bQ_{\geq 0 } \cdot(1,1)$ together with the induced general flag of type
$\mathbf{G_2}$:
\[Y_\bullet: \quad  \bP(\Omega_{\bP^2}) \supset  r\big(\pi^{-1}(Q) = \TV(\Sigma)\big) \supset r(D_{\delta_1}) \supset r(D_{\delta_1}
\cap D_{\delta_2}) = \xfix \,.\]
Applying Theorem \ref{thm:localOBgeneral}, we obtain that $\OB(-K_X) \subset \bR^3$ is the convex polytope with vertices represented
as the columns of the following matrix
\[\left(\begin{array}{ccccccc} 0 & 0 & 0 & 2 & 0 & 0 & 0 \\ 0 & 2 & 0 & 2 & 4 & 2 & 4 \\ 0 & 0 & 2 & 2 & 2 & 4 & 4 \end{array}\right).\]

\subsubsection{The Anti-Canonical Bundle on the Smooth Quadric}
\label{ex:quadricOB}
Recall from Example \ref{ex:quadric} that every maximal cone in the tailfan $\fan(Q)$ is marked and singular. Hence, we may only
construct a toric flag of type $\mathbf{T_2}$. 

So let us consider the p-divisor $\cD_{\sigmafix}$ whose tailcone $\sigmafix$ is generated by the rays $(1,1)$ and $(1,-1)$. Note
that the polyhedral coefficient $\cD_1$ is trivial. Hence, we have that $X(\cD_{\sigmafix}) = \TV(\deltafix) = \bA^3$ with
$\deltafix$ being spanned by the rays $\rho_1,\rho_2$ and $\rho_3$ whose primitive generators are $(1,0,0)$, $(1,0,-1)$ and $(-2,1,1)$,
respectively. We also take this enumeration for the definition of our flag, i.e.\
\[Y_\bullet: \quad \TV(\deltafix) \; \supset \; D_{\rho_1} \supset \; (D_{\rho_1} \cap D_{\rho_2}) \; \supset \; (D_{\rho_1}
\cap D_{\rho_2} \cap D_{\rho_3}) = \xfix \,.\]
An easy calculation now shows that
\[\OB(-K_X) = \conv\{(0,0,0),(3,0,0),(0,3,0),(0,0,6)\} \subset \bR^3\,,\]
whose normal fan corresponds to $\bP(1,1,2,2)$.

\section{Degenerations and Deformations}
\label{sec:ToricDeg}

\subsection{Anderson's Approach}
\label{subsec:anderson}
We will use our results from the previous section to investigate toric degenerations with a focus upon Dave Anderson's article
\cite{obToricDegenerations}. Let us first recall some notions and notation from loc.\ cit.

Let $K$ denote a field (which we always think of as a function field $\bK(X)$ over $\bK$) and equip $\bZ^d$ with the lexicographic
order. We fix a $\bZ^d$-valuation $\nu$ on $K$. For a finite-dimensional $\bK$-subspace $V \subset K$ we denote by $V^m \subset K$
the subspace which is spanned by elements of the form $f_1 \cdots f_m$ with $f_i \in V$. Furthermore, we set
\[\Gamma(V) \ldef \Gamma_\nu(V) \ldef \big \{(m,\nu(f)) \in \bN \times \bZ^d \lst f \in V^m \setminus 0 \big\} \subset \bN \times
\bZ^d\,,\]
which is a graded semigroup. By $\cone \Gamma(V) \subset \bR \times \bR^d$ we mean the closure of the convex hull of $\Gamma(V)$.

Following \cite{obKK}, one may define the \emph{Newton-Okounkov body} of $V$ as
\[\Delta(V) \ldef \Delta_\nu(V) \ldef \cone \Gamma(V) \cap \big(\{1\} \times \bR^d\big).\]

\begin{lemma}
\label{lem:finGenSG}
Let $(X,\cL) := (\TV(\Xi(\Psi)),\cO(D_{\Psi^*}))$ be a polarized rational $\bQ$-factorial projective $T$-variety of complexity one which
is associated to the divisorial polytope $(\Psi,\Box,\bP^1)$. Furthermore, fix a general or toric flag $Y_\bullet$ in $X$. As above,
we set $V := \Gamma(X,\cL)$ and $\nu := \nu_{Y_\bullet}$. Then the associated semigroup $\Gamma := \Gamma_{\nu}(V)$ is finitely
generated.
\end{lemma}

\begin{proof}
According to \cite[Lemma 2.2]{obToricDegenerations} it is enough to show that $\cone(\Gamma)$ is generated by $\Gamma \cap (\{1\} \times
\bZ^n)$. But since we are in the rational case $(C = \bP^1)$ it is not hard to see that $\nu_{Y_\bullet}(V^m \setminus 0) \subset
\bZ^d_{\geq0}$ consists exactly of the $m$-fold sums of elements in $\nu_{Y_\bullet}(V \setminus 0)$. Hence,
$\Delta(V) = \conv \big(\{1\} \times \nu_{Y_\bullet}(V \setminus 0)\big)$ which is a lattice polytope.
\end{proof}

\begin{proposition}
\label{prop:Tdegeneration}
Let $(X,\cL)$ be as in Lemma \ref{lem:finGenSG} and fix an integer $m>0$ such that $mL$ becomes very ample. The image of $X$ under the
linear system $|m\cL|$ then admits a flat degeneration to the not necessarily normal toric variety 
\[X(\Gamma) := \proj k[\Gamma]\] 
whose normalization is the toric variety associated to $\Delta(V)$.
\end{proposition}

\begin{proof}
See \cite[Theorem 5.4]{obToricDegenerations}. 
\end{proof}

We conclude this section with a result on degenerations of rational $\bQ$-factorial projective $\bK^*$-surfaces. Recall that a
point $v$ in a convex set $K$ in a real vector space $V$ is called \emph{extremal} if it does not lie on a compact line
segment contained in $K$.

\begin{proposition}
\label{prop:degenerationWPS}
Allowing for normalization after each degeneration step, every rational $\bQ$-factorial projective $\bK^*$-surface $X$ degenerates to
a weighted projective space.
\end{proposition}

\begin{proof}
We may assume that $X$ is toric with Picard rank $\geq 2$. Choosing an invariant ample divisor $D$ and considering $\partial P_D$ as
a circular graph whose vertices are the extremal points of $P_D$, we can find two extremal points $u_1,u_2 \in P_D$ of distance $2$,
i.e.\ there is exactly one extremal point in between. In the next step, we choose a primitive generator $v$ of $(u_1 - u_2)^\perp
\subset N$ to define a downgrade via the exact sequence 
\[\xymatrix{0 \ar[r] & (u_1 - u_2)^\perp \ar[r]^{\quad v} & N \ar[r] & N/(u_1 - u_2)^\perp \ar[r] & 0\,.}\]
It follows that there is at least one elliptic fixed point $\xfix$ (corresponding to the unique extremal point in between $u_1$ and
$u_2$). We continue by constructing a general flag $Y_\bullet$ of type $\mathbf{B_1}$ with respect to $\xfix$. The resulting Okounkov
body $\OB(D)$ then has one extremal point less than $P_D$ since $u_1$ and $u_2$ lie in the same fiber with respect to $v^\vee: M
\lra \Hom_\bZ \big( (u_1 - u_2)^\perp,\bZ\big)$. Hence, the Picard rank drops and we may proceed by induction.
\end{proof}

\begin{remark}
Proposition \ref{prop:degenerationWPS} can also be shown by using so-called degeneration diagrams, as presented in
\cite[Section 6.2]{phdNathan}. 
\end{remark}

\subsection{Ilten's Approach}
\label{subsec:ilten}

The previous section focused on degenerations. Now, we reverse our point of view and investigate the link between ``decompositions''
of Okounkov bodies and $T$-deformations. For the development and detailed treatment of the latter we refer the reader to
\cite{phdNathan}. In the following, we briefly recall the fundamental notion of Section 7.3 in loc.\ cit. 

We begin by recalling the notion of a decomposition of a divisorial polytope, cf. \cite[Definition 7.3.1]{phdNathan}. Let
$\Psi: \Box \to \wdiv_\bQ \bP^1$ be a divisorial polytope. An $\alpha$-\emph{admissible one-parameter decomposition} of $\Psi$
consists of two piecewise affine functions $\Psi_0^0, \Psi_0^1: \Box \to \bQ$ such that:
\begin{enumerate}
\item The graph of the map $\Psi^i_0$ has lattice vertices for $i=0,1$.
\item $\Psi_0(u) = \Psi_0^0(u) + \alpha\Psi^1_0(u)$ for all $u \in \Box$.
\item For any full-dimensional polyhedron in $\Box$ on which $\Psi_0$ is affine, $\Psi_0^i$ has non-integral slope on this polyhedron
for at most one $i \in \{0,1\}$.
\item If $\alpha \neq 1$, then $\Psi_0^1$ always has integral slope.
\end{enumerate}

\noindent
It is explained in loc.\ cit.\ how to construct a one-parameter $T$-deformation
\[\pi: \cX \to B\]
of $\TV(\Xi(\Psi))$ over the affine base $0 \in B \subset \bA^1 = \spec \bK[t]$ from such a decomposition of $\Psi$.

\begin{example}
We return to Example \ref{ex:hirzebruchDegenerationOB} to construct a $T$-deformation with special fiber $\bP(1,1,n+2)$ and general
fiber $\bF_n$ in terms of an $\alpha$-admissible one-parameter decomposition of $\OB(D_h)$, cf.\ \cite[Section 7.3]{phdNathan}. To do
so, we identify the Okounkov body with the divisorial polytope 
\[\Psi: [0,n] \to \wdiv_\bQ \bP^1\,, \textnormal{ with}\; \left\{\begin{array}{rcll} \Psi_P(u) & = & \hstar_0(u) + \hstar_\infty(u),
& P \textnormal{ a fixed point},\\[1mm] \Psi_Q & \equiv & 0\,, & Q \in \bP^1 \setminus P\,. \end{array}\right. \] 
It is not hard to see that $\Psi$ yields $\bP(1,1,n+2)$ together with the ample line bundle $\cO(D_{\Psi^\ast})$.
\end{example}

\begin{example}
Let us consider the Okounkov body $\OB(-K_Q)$ of the (ample) anti-canonical bundle on the smooth quadric in $\bP^4$ with respect to a
toric flag of type $\mathbf{T_2}$. It is now not hard to construct two decompositions of the latter such that there is a $T$-deformation
with special fiber $X_\ast = \bP(1,1,2,2)$ and general fiber $Q$ which arises from the concatenation of these two decompositions.
\end{example}

\section{Global Okounkov Bodies}
\label{sec:GlobalOB}

\noindent
Before stating the main theorem of this section, we give a short description of the pseudo-effective cone $\psEff(\TV(\fan))$ of a
rational projective complexity-one $T$-variety $\TV(\fan)$.

As in toric geometry, there is an exact sequence describing the divisor class group $\cl(\TV(\fan))$ of a complete rational
complexity-one $T$-variety $\TV(\fan)$. We denote by $\cP \subset \bP^1$ the set of points with non-trivial slices $\fan_P$. Then we
have that
\[\xymatrix{ 0 \ar[r] & (\bZ^\cP/\bZ)^\vee \oplus M  \ar[r]^{\iota} & \tdiv\big(\TV(\fan)\big) \ar[r]^{\textnormal{pr}} & \cl(\TV(\fan))
\ar[r] & 0 \,,} \]
where $\tdiv\big(\TV(\fan)\big) \cong (\bZ^{\cV \cup \cR})^\vee$, cf.\ (\ref{subsec:divCL}). Thus, we see that
\[\psEff(\TV(\fan)) = \textnormal{pr}(\bR^{\cV \cup \cR}_{\geq 0}) \subset \cl(\TV(\fan))_\bR\,.\]
By definition it is rational polyhedral. Indeed, after choosing a basis in every of these lattices, the maps $\iota$ and
$\textnormal{pr}$ become integer matrices.

\begin{theorem}
\label{thm:globalOB}
The global Okounkov body $\Delta_{Y_\bullet}(\TV(\fan))$ of a rational projective complexity-one $T$-variety $\TV(\fan)$ with respect
to a general or toric flag $Y_\bullet$ is a rational polyhedral cone.
\end{theorem}

We postpone the proof until (\ref{subsec:mainResult}).

\subsection{A Lemma on Polyhedra}
\label{subsec:polLemma}
Let $v \in \bR^k$ be a point, $b = (b_1,\dots,b_k)$ an orthonormal basis of $\bR^k$, and $\lambda = (\lambda_1,\dots,\lambda_k) \in
\bR^k_{>0}$. Using these data, we construct a piecewise linear object $T(v,b,\lambda) \subset \bR^k$ in the following way:
\[ \begin{array}{ccl} T_1 & = & \{v + \kappa b_1 \;|\; 0\leq \kappa \leq \lambda_1\}\, \\[1mm]
T_i & = & T_{i-1} \cup \{v + \sum_{j=1}^{i-1} \frac{\lambda_j}{2} b_j + \kappa b_i \;|\; 0 \leq \kappa \leq \lambda_i\} \,. \end{array}\]
Finally, we arrive at $T_k =: T(v,b,\lambda)$. It is now not hard to see that a concave function $f: \bR^k \to \bR$ which is affine
linear on $T(v,b,\lambda)$ is also affine linear on the convex hull $\conv T(v,b,\lambda)$.

\begin{definition}
Consider a set $T := T(v,b,\lambda) \subset \bR^k$ as constructed above, and define
\[\bS^{k-1}_T := \frac{1}{\epsilon} \big( \conv T \cap  \bS^{k-1}(v,\epsilon) \big)\,,\]
where $\bS^{k-1}(v,\epsilon)$ denotes the sphere of radius $0 < \epsilon \ll 1$ around $v \in \bR^k$. The sphere is supposed to be
small enough such that no other extremal point of $\conv T$ apart from $v$ is contained in the ball $B(v,\epsilon)$.
\end{definition}

\begin{lemma}
\label{lem:fiberedPolytopes}
Let $P \subset \bR^k$ be a $k$-dimensional polytope and $f: P \to \bR_{\geq 0}$ a non-negative concave function. Assume that the set 
\[Q_S := \{(x,y)\;|\; x \in S, \, 0 \leq y \leq f(x)\} \subset \bR^{k+1}\] 
over any line segment $S \subset P$ is a polytope. Then 
\[Q := \{(x,y)\;|\; x \in P, \, 0 \leq y \leq f(x) \} \subset \bR^{k+1}\]
is also a polytope.
\end{lemma}

Crucial input in the following proof was provided by Christian Haase.

\begin{proof}
An element $v$ of $P$ is called a \emph{vertex} if and only if $(v,f(v))$ is an extremal point of $Q_S$ for all line segments
$S \subset P$ containing $v$. Then $Q$ will be the convex hull of 
\[V := \{(v,f(v))\,|\; v \textnormal { vertex} \} \cup (P \times \{0\})\,,\]
since $Q$ is the convex hull of its extremal points which are, by definition, all contained in $V$. To see this, recall that a point
of a convex set is called extremal if it does not lie in the middle of a compact line segment contained within this set.

We are left to show that the set of vertices is isolated in $Q$. So let $v \in P$ be a vertex, and denote by $L_v \subset \bS^{k-1}$ the
compact subset of directions from $v$ which see other points of $P$, and define $S_l : = P \cap \{v + \kappa l \;|\; 0 \leq \kappa < \infty \}$
for $l \in L_v$. Next we set
\[r_v(l) := \min \{\alpha >0 \;|\; (v+\alpha l,f(v + \alpha l)) \neq (v,f(v)) \textnormal{ is an extremal point of } Q_{S_l} \}\,. \]
Showing that there is an $\epsilon >0$ such that $r_v(l) \geq \epsilon$ for all $l \in L_v$ will complete the proof, since vertices
must appear as extremal points on some $Q_{S_l}$. So let us consider a direction $l_1 \in L_v$, and set $\lambda_1 = r_v(l)$. From
$v_1 := v + \frac{\lambda_1}{2} l$ we can proceed along a direction $l_2 \in (l_1)^\perp$ to set $\lambda_2 := r_{v_1}(l_2)$, and
$v_2 = v_1 + \frac{\lambda_2}{2}l_2$. Again, we can walk along a direction $l_3 \in (\bangle{l_1}+\bangle{l_2})^\perp$ with
$\lambda_3 = r_{v_2}(l_3)$. Iterating this procedure, we finally obtain an object $T(l) = T(v,\lambda(l))$ on which $f$ will be affine
linear. So by the discussion from above it will be affine linear on the whole $k$-dimensional polytope $\conv T(l)$. 

Doing this for all elements $l \in L_v$ gives us an infinite covering of $L_v$ by $\bS^{k-1}_{T(l)}$. Since $L_v$ is compact we can
choose a finite number of these to provide a covering. Therefore we can also find an $\epsilon >0$ such that there is no vertex
$v' \in P$ with $d\big((v,f(v)),(v',f(v'))\big) < \epsilon$.
\end{proof}

\subsection{Proof of Theorem \ref{thm:globalOB}}
\label{subsec:mainResult}

\begin{remark}
Before going into the details, we would like to make some identifications.

Elements of the rational pseudo-effective cone $\psEff(\TV(\fan))_\bQ := \psEff(\TV(\fan)) \cap N^1(\TV(\fan))_\bQ$ will be denoted
by $\xi$. Having fixed a general or toric flag $Y_\bullet$ in $\TV(\fan)$, the rational polytopes $\Box_{\ul{\xi}}$ and $W(\xi)$ are
defined as $\Box_{\ul{h}}$ and $W(h)$, respectively, for the unique normalized $\bQ$-Cartier divisor $D_h$ with $[D_h] = \xi$. In
the same vein, for every $P \in \bP^1$ we define the map $\xi^*_P: \Box_{\ul{\xi}} \to \bQ$ as the map $\hstar_P: \Box_{\ul{h}} \to \bQ$.
Apart from these identifications we will also make use of the following linear map
\[\gamma: N^1(\TV(\fan))_\bQ \to N^1(\TV(\tail \fan))_\bQ, \quad \xi \mapsto \gamma(\xi) = \ul{\xi}\,,\]
whose image of $\psEff(\TV(\fan))_\bQ$ lies inside $\psEff(\TV(\tail \fan))_\bQ$.
\end{remark}

Let us now proceed to the proof.\\[1ex]

\noindent
Let $Y_\bullet$ be a general flag and denote by $E \subset \bQ^{d-1} \oplus N^1(\TV(\fan))_\bQ$ the cone over $\psEff(\TV(\fan))_\bQ$
with fiber $\phi_\bQ(\Box_{\ul{\xi}})$. Its closure $\ovl{E}$ in $\bR^{d-1} \oplus N^1(\TV(\fan))$ is rational polyhedral, since
$\psEff(\TV(\fan))$ is rational polyhedral ($\TV(\fan)$ is a Mori dream space) and $\ovl{E}$ arises as the pull-back of the global
Okounkov body $\Delta_{Y_\geq 1}\big(\TV(\tail \fan))\big)$ along $\id_{\bR^{d-1}} \oplus \gamma_\bR$. By forgetting the first
coordinate, we get a projection 
\[p: \Delta_{Y_\bullet}(\TV(\fan)) \lra \bR^{d-1} \oplus N^1(\TV(\fan))_\bR \] 
with image exactly equal to $\ovl{E}$. Moreover, one can reconstruct $\OB(\TV(\fan))$ from $E$ by considering the graph of $\hstar$.
Namely, we have that  $\OB(\TV(\fan))$ is equal to the closure of
\[\bigg\{(x,u,\xi) \in \bQ \oplus \bQ^{d-1} \oplus N^1(\TV(\fan))_\bQ \,\bigg| \begin{array}{l}  (u,\xi) \in E \subset \bQ^{d-1} \oplus
N^1(\TV(\fan))_\bQ \,, \\[1mm] 0 \leq x \leq \sum_{P \in \cP} \xi^*_P(u) \end{array} \bigg\}\]
inside $\bR \oplus \bR^{d-1} \oplus N^1(\TV(\fan))_\bR$. Observe that the map
\[\mathbf{h}_{\psEff}: \bQ^{d-1} \oplus N^1(\TV(\fan))_\bQ \supset E \lra \bQ \,, \; (u,\xi) \mapsto \sum_{P \in \bP^1} \xi^*_P(u) = 
\deg \xi^*(u)\]
is concave and linear on rays, i.e.\ 
\[\mathbf{h}_{\psEff}(\lambda \cdot (u,\xi)) = \lambda \cdot \mathbf{h}_{\psEff}(u,\xi)\,, \quad \lambda \geq 0.\]
We claim that this map varies piecewise affine linearly along any compact line segment 
\[S(c_1,c_2) = \{\lambda c_1 + (1-\lambda)c_2 \;|\; 0 \leq \lambda \leq 1\} \subset E\] 
between two distinct points $c_1 = (u_1,\xi_1), \; c_2 = (u_2,\xi_2) \in E \subset \bQ^{d-1} \oplus N^1(\TV(\fan))_\bQ$ with only a
finite number of breaks in the linear structure. Note that it is enough to check this for a single summand 
\[(\lambda\xi_1 + (1-\lambda) \xi_2)^*_P(\lambda u_1 + (1-\lambda) u_2)\]
for an arbitrary but fixed point $P \in \bP^1$. Recall that
\[\hstar_P(u) = \min \{u(v)-h_P(v) \lst v \in \fan_P(0)\}.\]
where $u-h_P$ is a piecewise affine linear function on $N_\bQ$ and $\fan_P(0)$ is a finite set.
Note that there exists a real number $\epsilon > 0$ such that, for $0 \leq \lambda,\mu \leq 1$, the functions $\lambda u_1 +
(1-\lambda)u_2 - (\lambda (\xi_1)_P + (1-\lambda)(\xi_2)_P)$ and $\mu u_1 + (1-\mu)u_2 - (\mu (\xi_1)_P + (1-\mu)(\xi_2)_P)$ attain
their minimum at the same vertex $v \in \fan_P(0)$ whenever $|\lambda - \mu| < \epsilon$. Hence, we can partition the line segment
$S(c_1,c_2)$ into a finite number of segments along which $(\lambda\xi_1 + (1-\lambda) \xi_2)^*_P(\lambda u_1 + (1-\lambda) u_2)$ is
in fact affine linear. Taking a rational polytopal cross section of the cone $\psEff(\TV(\fan))$ and applying Lemma
\ref{lem:fiberedPolytopes} then shows that a cross section of $\Delta_{Y_\bullet}(\TV(\fan))$ is a rational polytope. Since
$\Delta_{Y_\bullet}(\TV(\fan))$ arises as the cone over this rational polytopal cross section it has to be rational polyhedral, too.
\\[.5ex]
Finally, let $Y_\bullet$ be a toric flag and denote by $E \subset \bQ^{d-1} \oplus N^1(\TV(\fan))_\bQ$ the cone over
$\psEff(\TV(\fan))_\bQ$ with fiber $W(\xi)$ for $\xi \in \psEff(\TV(\fan))_\bQ$. We only have to show that the closure
$\ovl{E} \subset \bR^{d-1} \oplus N^1(\TV(\fan))_\bR$ is rational polyhedral since $\OB(\TV(\fan))$ arises as a translation of
$\ovl{E}$ induced by the primitive generators of the rays of $\deltafix$, cf.\ Propositions \ref{prop:localOBtoricATwo} and
\ref{prop:localOBtoricBTwo}. But this claim follows easily from the explicit description of the rational polytope $W(h)$ and
the general arguments concerning the piecewise affine structure of ``\;$\cdot^*_P$\;'' as a function in $\xi$ we have given above
in the first part of the proof.

\begin{remark}
Theorem \ref{thm:globalOB} shows that the global Okounkov body of a rational projective complexity-one $T$-variety $\TV(\fan)$ is
determined by the global Okounkov body of the general fiber $\TV(\tail \fan)$ and the function ``\;${}^*$\;'' which maps an element
$\xi \in N^1(\TV(\fan))_\bQ$ to $\xi^*$.

Our result generalizes Theorem 5.2 from \cite{obToricVB} which states that the global Okounkov body of a rank two toric vector bundle
on a smooth projective toric variety with respect to a toric flag of type $\mathbf{T_2}$ (cf.\ Remark \ref{rmk:flagsGonzalez}) is
rational polyhedral. However, Theorem \ref{thm:globalOB} does not give us any explicit equations for $\OB(\TV(\fan))$ as they were
obtained in loc.\ cit.
\end{remark}

\bibliographystyle{alpha}
\bibliography{okounkovBody}

\end{document}

%% file: figuresOB.tex
\newgray{gray1}{0.4}
\newgray{gray2}{0.5}
\newgray{gray3}{0.6}
\newgray{gray4}{0.7}
\newgray{gray5}{0.8}
\newgray{gray6}{0.9}
\newgray{gray7}{0.3}

\newcommand{\cotangfannullvv}{%
\psset{xunit=0.35cm,yunit=.35cm}
\begin{pspicture}(-3.2,-3.2)(3.2,3.2)%

\pspolygon[linewidth=.001pt,linecolor=white,fillstyle=solid,fillcolor=gray1](0,3)(0,1)(2,3)%
\pspolygon[linewidth=.001pt,linecolor=white,fillstyle=solid,fillcolor=gray3](3,3)(2,3)(0,1)(0,0)(3,0)%
\pspolygon[linewidth=.001pt,linecolor=white,fillstyle=solid,fillcolor=gray5](3,0)(0,0)(0,-3)(3,-3)%
\pspolygon[linewidth=.001pt,linecolor=white,fillstyle=solid,fillcolor=gray4](0,-3)(0,0)(-3,-3)%
\pspolygon[linewidth=.001pt,linecolor=white,fillstyle=solid,fillcolor=gray2](-3,-3)(0,0)(0,1)(-3,1)%
\pspolygon[linewidth=.001pt,linecolor=white,fillstyle=solid,fillcolor=gray6](-3,1)(0,1)(0,3)(-3,3)%
\psset{linewidth=1pt}%

\psline{-}(0,3)(0,1)(2,3)%
\psline{-}(2,3)(0,1)(0,0)(3,0)%
\psline{-}(3,0)(0,0)(0,-3)%
\psline{-}(0,-3)(0,0)(-2,-2)%
\psline{-}(-3, -3)(0,0)(0,1)(-3,1)%
\psline{-}(-3, 1)(0,1)(0,3)%
\psgrid[gridwidth=0.3pt,griddots=5,subgriddiv=1,gridlabels=5pt](-3,-3)(3,3)
\end{pspicture}}

\newcommand{\cotangfaninftyvv}{%
\psset{xunit=0.35cm,yunit=.35cm}
\begin{pspicture}(-3.2,-3.2)(3.2,3.2)%
\pspolygon[linewidth=.001pt,linecolor=white,fillstyle=solid,fillcolor=gray1](0,3)(0,0)(3,3)%
\pspolygon[linewidth=.001pt,linecolor=white,fillstyle=solid,fillcolor=gray3](3,3)(0,0)(3,0)%
\pspolygon[linewidth=.001pt,linecolor=white,fillstyle=solid,fillcolor=gray5](3,0)(0,0)(-1,-1)(-1,-3)(3,-3)%
\pspolygon[linewidth=.001pt,linecolor=white,fillstyle=solid,fillcolor=gray4](-1,-3)(-1,-1)(-3,-3)%
\pspolygon[linewidth=.001pt,linecolor=white,fillstyle=solid,fillcolor=gray2](-3,-3)(-1,-1)(-3,-1)%
\pspolygon[linewidth=.001pt,linecolor=white,fillstyle=solid,fillcolor=gray6](-3,-1)(-1,-1)(0,0)(0,3)(-3,3)%

\psset{linewidth=1pt}%
\psline{-}(0,3)(0,0)(3,3)%
\psline{-}(3,3)(0,0)(3,0)%
\psline{-}(3,0)(0,0)(-1,-1)(-1,-3)%
\psline{-}(-1,-3)(-1,-1)(-2,-2)%
\psline{-}(-3, -3)(-1,-1)(-3,-1)%
\psline{-}(-3, -1)(-1,-1)(0,0)(0,3)%
\psgrid[gridwidth=0.3pt,griddots=5,subgriddiv=1,gridlabels=5pt](-3,-3)(3,3)
\end{pspicture}}

\newcommand{\cotangfanonevv}{%
\psset{xunit=0.35cm,yunit=.35cm}
\begin{pspicture}(-3.2,-3.2)(3.2,3.2)%
\pspolygon[linewidth=.001pt,linecolor=white,fillstyle=solid,fillcolor=gray1](0,3)(0,0)(1,0)(3,2)(3,3)%
\pspolygon[linewidth=.001pt,linecolor=white,fillstyle=solid,fillcolor=gray3](3,2)(1,0)(3,0)%
\pspolygon[linewidth=.001pt,linecolor=white,fillstyle=solid,fillcolor=gray5](3,0)(1,0)(1,-3)(3,-3)%
\pspolygon[linewidth=.001pt,linecolor=white,fillstyle=solid,fillcolor=gray4](1,-3)(1,0)(0,0)(-3,-3)%
\pspolygon[linewidth=.001pt,linecolor=white,fillstyle=solid,fillcolor=gray2](-3,-3)(0,0)(-3,0)%
\pspolygon[linewidth=.001pt,linecolor=white,fillstyle=solid,fillcolor=gray6](-3,0)(0,0)(0,3)(-3,3)%

\psset{linewidth=1pt}%
\psline{-}(0,3)(0,0)(1,0)(3,2)%
\psline{-}(3,2)(1,0)(3,0)%
\psline{-}(3,0)(1,0)(1,-3)%
\psline{-}(1,-3)(1,0)(0,0)(-3,-3)%
\psline{-}(-3,-3)(0,0)(-3,0)%
\psline{-}(-3,0)(0,0)(0,3)%
\psgrid[gridwidth=0.3pt,griddots=5,subgriddiv=1,gridlabels=5pt](-3,-3)(3,3)
\end{pspicture}}

\newcommand{\tailFan}{%
\psset{xunit=0.35cm,yunit=.35cm}
\begin{pspicture}(-3.2,-3.2)(3.2,3.2)%
\pspolygon[linewidth=.001pt,linecolor=white,fillstyle=solid,fillcolor=gray1](0,3)(0,0)(3,3)%
\pspolygon[linewidth=.001pt,linecolor=white,fillstyle=solid,fillcolor=gray3](3,3)(0,0)(3,0)%
\pspolygon[linewidth=.001pt,linecolor=white,fillstyle=solid,fillcolor=gray5](3,0)(0,0)(0,-3)(3,-3)%
\pspolygon[linewidth=.001pt,linecolor=white,fillstyle=solid,fillcolor=gray4](0,-3)(0,0)(-3,-3)%
\pspolygon[linewidth=.001pt,linecolor=white,fillstyle=solid,fillcolor=gray2](-3,-3)(0,0)(-3,0)%
\pspolygon[linewidth=.001pt,linecolor=white,fillstyle=solid,fillcolor=gray6](-3,0)(0,0)(0,3)(-3,3)%

\psset{linewidth=1pt}%
\psline{-}(0,3)(0,-3)%
\psline{-}(-3,0)(3,0)%
\psline{-}(-3,-3)(3,3)%
\psgrid[gridwidth=0.3pt,griddots=5,subgriddiv=1,gridlabels=5pt](-3,-3)(3,3)
\end{pspicture}}

\newcommand{\degreeFan}{%
\psset{xunit=0.35cm,yunit=.35cm}
\begin{pspicture}(-3.2,-3.2)(3.2,3.2)%
\pspolygon[linewidth=.001pt,linecolor=white,fillstyle=solid,fillcolor=gray1](0,3)(0,1)(1,1)(3,3)%
\pspolygon[linewidth=.001pt,linecolor=white,fillstyle=solid,fillcolor=gray3](3,3)(1,1)(1,0)(3,0)%
\pspolygon[linewidth=.001pt,linecolor=white,fillstyle=solid,fillcolor=gray5](3,0)(1,0)(0,-1)(0,-3)(3,-3)%
\pspolygon[linewidth=.001pt,linecolor=white,fillstyle=solid,fillcolor=gray4](0,-3)(0,-1)(-1,-1)(-3,-3)%
\pspolygon[linewidth=.001pt,linecolor=white,fillstyle=solid,fillcolor=gray2](-3,-3)(-1,-1)(-1,0)(-3,0)%
\pspolygon[linewidth=.001pt,linecolor=white,fillstyle=solid,fillcolor=gray6](-3,0)(-1,0)(0,1)(0,3)(-3,3)%
\pspolygon[linewidth=.001pt,linecolor=white,fillstyle=crosshatch,hatchsep=0.1,hatchangle=0](0,1)(1,1)(1,0)(0,-1)(-1,-1)(-1,0)%

\psset{linewidth=1pt}%
\psline{-}(0,1)(1,1)(1,0)(0,-1)(-1,-1)(-1,0)(0,1)%
\psline{-}(0,1)(0,3)
\psline{-}(1,1)(3,3)%
\psline{-}(1,0)(3,0)%
\psline{-}(0,-1)(0,-3)%
\psline{-}(-1,-1)(-3,-3)%
\psline{-}(-1,0)(-3,0)%
\psgrid[gridwidth=0.3pt,griddots=5,subgriddiv=1,gridlabels=5pt](-3,-3)(3,3)
\end{pspicture}}

\newcommand{\boxCotang}{%
\psset{xunit=.5cm,yunit=.4cm}
\begin{pspicture}(-3,-3)(3,3)%
\pspolygon[linewidth=1pt,fillstyle=solid,fillcolor=gray5](0,-2)(2,-2)(2,0)(0,2)(-2,2)(-2,0)
\psgrid[gridwidth=0.3pt,griddots=7,subgriddiv=1,gridlabels=5pt](0,0)(-3,-3)(3,3)
\end{pspicture}}

\newcommand{\TDhirzebruch}{%
\psset{unit=0.8cm}
\begin{pspicture}(0,-6)(12,0)
\psframe[linecolor=white](0.5,-4.5)(3.5,-1.5)

\psline{->}(2,-3)(4,-3)
\psline{->}(2,-3)(0,-1)
\psline{<->}(2,-5)(2,-1)
\psline[linewidth=0.5pt, linestyle=dotted]{-|}(0,-1.8)(0.8,-1.8)
\psline[linewidth=0.5pt,linestyle=dotted]{|-|}(0.8,-1.8)(2,-1.8)
\psline[linewidth=0.5pt, linestyle=dotted]{|-}(2,-1.8)(4,-1.8)
\psline[linewidth=0.5pt, linestyle=dotted]{-|}(0,-4.2)(2,-4.2)
\psline[linewidth=0.5pt, linestyle=dotted]{|-}(2,-4.2)(4,-4.2)

\qdisk(0.5,-1.5){1pt}
\rput[bl]{0}(2.7,-2.15){$\sigma_0$}
\rput[bl]{0}(1.4,-2.15){$\sigma_1$}
\rput[bl]{0}(0.8,-3.3){$\sigma_2$}
\rput[bl]{0}(2.7,-4){$\sigma_3$}
\rput[tl]{0}(3.8,-3.1){$\rho_0$}
\rput[bl]{0}(2.1,-1.5){$\rho_1$}
\rput[tr]{0}(0,-1){$\rho_2$}
\rput[br]{0}(1.9,-5){$\rho_3$}
\rput[tr]{0}(1.4,-1.2){\tiny{$(-1,n)$}}
\rput[bl]{0}(2,-6){$\bF_n$}

\psline{<-}(6,-3)(8,-3)
\psline{|->}(8,-3)(10,-3)
\rput[bl]{0}(10.5,-3){\textnormal{tailfan}}

\psline{<-|}(6,-1.8)(7.25,-1.8)
\psline{-|}(7.25,-1.8)(8,-1.8)
\psline{->}(8,-1.8)(10,-1.8)
\uput*[270](7.25,-1.8){{\tiny $-\frac{1}{n}$}}
\rput[bl]{0}(10.8,-1.8){$\fan_0$}
\rput[bl]{0}(9,-1.6){$\cD^{\sigma_0}$}
\rput[bl]{0}(6.4,-1.6){$\cD^{\sigma_2}$}
\rput[bl]{0}(7.3,-1.6){$\cD^{\sigma_1}$}

\psline{<-|}(6,-4.2)(8,-4.2)
\psline{->}(8,-4.2)(10,-4.2)
\rput[bl]{0}(10.8,-4.2){$\fan_{\infty}$}
\rput[bl]{0}(9,-4){$\cD^{\sigma_3}$}
\rput[bl]{0}(6.5,-4){$\cD^{\sigma_2}$}
\rput[bl]{0}(8,-6){$\fan$}

\psline{<->}(5,-5)(5,-1)
\qdisk(5,-1.8){1pt}
\qdisk(5,-4.2){1pt}
\qdisk(5,-3){1.5pt}
\rput[bl]{0}(4.5,-6){$Y=\bP^1$}

\end{pspicture}}

\newcommand{\hstarzero}{%
\psset{xunit=1cm,yunit=0.75cm}
\begin{pspicture}(0,-1.5)(3,1.5)%
\psgrid[gridwidth=0.3pt,griddots=5,subgriddiv=1,gridlabels=5pt](0,-1.5)(3,1.5)
\psline[linewidth=1.5pt]{-}(0,0)(1,0)
\psline[linewidth=1.5pt]{-}(1,0)(3,-1)
\end{pspicture}}

\newcommand{\hstarinfty}{%
\psset{xunit=1cm,yunit=0.75cm}
\begin{pspicture}(0,-1.5)(3,1.5)%
\psgrid[gridwidth=0.3pt,griddots=5,subgriddiv=1,gridlabels=5pt](0,-1.5)(3,1.5)
\psline[linewidth=1.5pt]{-}(0,1)(3,1)
\end{pspicture}}

\newcommand{\hirzebruchTwoPolytope}{%
\psset{xunit=.9cm,yunit=.75cm}
\begin{pspicture}(0,-1)(4,2)%
\pspolygon[linewidth=1.5pt,fillstyle=solid,fillcolor=gray5](0,0)(1,0)(3,1)(0,1)%
\psgrid[gridwidth=0.1pt,subgriddiv=1,griddots=5,gridlabels=5pt](0,-1)(4,2)%
\end{pspicture}}

\newcommand{\hirzebruchNPolytope}{%
\psset{xunit=1cm}
\begin{pspicture}(-1,-1)(4,2)%
\pspolygon[linewidth=1.5pt,fillstyle=solid,fillcolor=gray5](0,0)(1,0)(3,1)(0,1)%
\rput[tr]{0}(-.1,0){$(0,0)$}
\rput[br]{0}(-.1,1){$(0,1)$}
\rput[bl]{0}(3.1,1){$(n+1,1)$}
\rput[tl]{0}(1.1,0){$(1,0)$}
\end{pspicture}}

\newcommand{\hirzebruchNOB}{%
\psset{xunit=1cm}
\begin{pspicture}(-1,-1)(4,2)%
\pspolygon[linewidth=1.5pt,fillstyle=solid,fillcolor=gray5](0,0)(1,1)(3,0)%
\rput[tr]{0}(-.1,0){$(0,0)$}
\rput[br]{0}(0.9,1){$(1,1)$}
\rput[bl]{0}(3.1,0){$(n+2,0)$}
\end{pspicture}}

\newcommand{\hirzebruchHstarZero}{%
\psset{xunit=1cm,yunit=0.75cm}
\begin{pspicture}(0,-1)(4,2)%
\psgrid[gridwidth=0.3pt,griddots=5,subgriddiv=1,gridlabels=0pt](0,-1)(4,2)
\psline[linewidth=1.5pt]{-}(0,0)(1,1)
\psline[linewidth=1.5pt]{-}(1,1)(4,1)
\rput[br]{0}(-.2,0){{\tiny{$(0,0)$}}}
\rput[br]{0}(.8,1.1){{\tiny{$(1,1)$}}}
\rput[bl]{0}(4.1,1){{\tiny{$(n,1)$}}}

\end{pspicture}}

\newcommand{\hirzebruchHstarInfty}{%
\psset{xunit=1cm,yunit=0.75cm}
\begin{pspicture}(0,-1)(4,2)%
\psgrid[gridwidth=0.3pt,griddots=5,subgriddiv=1,gridlabels=0pt](0,-1)(4,2)
\psline[linewidth=1.5pt]{-}(0,0)(1,0)
\psline[linewidth=1.5pt]{-}(1,0)(4,-1)
\rput[br]{0}(-.2,0){{\tiny{$(0,0)$}}}
\rput[bl]{0}(1.1,0.1){{\tiny{$(1,0)$}}}
\rput[bl]{0}(4.1,-1){{\tiny{$(n,-1)$}}}
\end{pspicture}}

\newcommand{\hirzHstarZero}{%
\psset{xunit=1cm,yunit=0.75cm}
\begin{pspicture}(0,0)(4,3)%
\psgrid[gridwidth=0.3pt,griddots=5,subgriddiv=1,gridlabels=0pt](0,0)(4,3)
\psline[linewidth=1.5pt]{-}(0,1)(1,1)
\psline[linewidth=1.5pt]{-}(1,1)(4,0)
\rput[br]{0}(-.2,1){{\tiny{$(-1,0)$}}}
\rput[bl]{0}(1.1,1.1){{\tiny{$(0,0)$}}}
\rput[bl]{0}(4.1,0){{\tiny{$(n,-1)$}}}
\end{pspicture}}

\newcommand{\hirzHstarInfty}{%
\psset{xunit=1cm,yunit=0.75cm}
\begin{pspicture}(0,0)(4,3)%
\psgrid[gridwidth=0.3pt,griddots=5,subgriddiv=1,gridlabels=0pt](0,0)(4,3)
\psline[linewidth=1.5pt]{-}(0,2)(4,2)
\rput[br]{0}(-.2,2){{\tiny{$(-1,1)$}}}
\rput[bl]{0}(4.1,2){{\tiny{$(n,1)$}}}
\end{pspicture}}

\newcommand{\toricSurfaceSimpleFan}{%
\psset{xunit=.7cm,yunit=.5cm}%
\begin{pspicture}(-4,-3)(4,3)%

\psline{->}(-.5,2)(-3,2)%
\psline{|->}(-.5,2)(3,2)%
\uput*[270](-.5,1.8){\tiny{$-\frac{1}{2}$}}%
\rput[bl](3.5,1.8){$\fan_\infty$}%

\psline{<-}(-3,-2)(.5,-2)%
\psline{|->}(.5,-2)(3,-2)%
\uput*[270](.5,-2.2){\tiny{$\frac{1}{2}$}}%
\rput[bl](3.5,-2.2){$\fan_0$}%

\end{pspicture}}

\newcommand{\toricSurfaceSimplePlus}{%
\psset{xunit=.7cm,yunit=.5cm}%
\begin{pspicture}(-4,-3)(4,3)%

\psline{<-|}(-3,1)(0,1)%
\psline{->}(0,1)(3,1)%
\rput[bl](3.5,0.8){$\tail \fan$}%

\uput*[270](0,-.5){$\deg \fan = \emptyset$}%

\end{pspicture}}

\newcommand{\hirzebruchVarFan}{%
\psset{xunit=.7cm,yunit=.5cm}%
\begin{pspicture}(-4,-3)(4,3)%

\psline{->}(-.5,2)(-3,2)%
\psline{|-|}(-.5,2)(0,2)%
\psline{->}(0,2)(3,2)%
\uput*[270](-.7,1.8){\tiny{$-\frac{1}{n}$}}%
\uput*[270](0,1.8){\tiny{$0$}}
\rput[bl](3.5,1.8){$\fan_\infty$}%

\psline{<-}(-3,-2)(0,-2)%
\psline{|-|}(0,-2)(1,-2)%
\psline{->}(1,-2)(3,-2)%
\uput*[270](0,-2.2){\tiny{$0$}}%
\uput*[270](1,-2.2){\tiny{$1$}}%
\rput[bl](3.5,-2.2){$\fan_0$}%

\end{pspicture}}

\newcommand{\hirzebruchVarPlus}{%
\psset{xunit=.7cm,yunit=.5cm}%
\begin{pspicture}(-4,-3)(4,3)%

\psline{<-|}(-3,1)(0,1)%
\psline{->}(0,1)(3,1)%
\rput[bl](3.5,0.8){$\tail \fan$}%

\psline{<-|}(-3,-1)(-.5,-1)%
\psline[linestyle=dotted]{-}(-.5,-1)(1,-1)%
\psline{|->}(1,-1)(3,-1)%
\uput*[270](-.5,-1.2){\tiny{$-\frac{1}{n}$}}%
\uput*[270](1,-1.2){\tiny{$1$}}%
\rput[bl](3.5,-1.2){$\deg\fan$}%

\end{pspicture}}

\newcommand{\quadricFanOne}{%
\psset{xunit=0.4cm,yunit=.4cm}
\begin{pspicture}(-3.2,-3.2)(3.2,3.2)%
\pspolygon[linewidth=.001pt,linecolor=white,fillstyle=solid,fillcolor=gray1](2,2)(0,0)(2,-2)%
\pspolygon[linewidth=.001pt,linecolor=white,fillstyle=solid,fillcolor=gray3](2,2)(0,0)(-1,0)(-3,2)%
\pspolygon[linewidth=.001pt,linecolor=white,fillstyle=solid,fillcolor=gray5](-3,2)(-1,0)(-3,-2)%
\pspolygon[linewidth=.001pt,linecolor=white,fillstyle=solid,fillcolor=gray4](-3,-2)(-1,0)(0,0)(2,-2)%
\psset{linewidth=1pt}%
\psline{-}(2,2)(0,0)(-1,0)(-3,2)%
\psline{-}(2,-2)(0,0)(-1,0)(-3,-2)%
\psgrid[gridwidth=0.3pt,griddots=5,subgriddiv=1,gridlabels=5pt](-3,-2)(2,2)
\end{pspicture}}

\newcommand{\quadricFanZero}{%
\psset{xunit=0.4cm,yunit=.4cm}
\begin{pspicture}(-3.2,-3.2)(3.2,3.2)%
\pspolygon[linewidth=.001pt,linecolor=white,fillstyle=solid,fillcolor=gray1](2,2)(0,0)(0,-1)(2,-3)%
\pspolygon[linewidth=.001pt,linecolor=white,fillstyle=solid,fillcolor=gray3](2,2)(0,0)(-2,2)%
\pspolygon[linewidth=.001pt,linecolor=white,fillstyle=solid,fillcolor=gray5](-2,2)(0,0)(0,-1)(-2,-3)%
\pspolygon[linewidth=.001pt,linecolor=white,fillstyle=solid,fillcolor=gray4](-2,-3)(0,-1)(2,-3)%

\psset{linewidth=1pt}%
\psline{-}(2,2)(0,0)(0,-1)(2,-3)%
\psline{-}(-2,2)(0,0)(0,-1)(-2,-3)%
\psgrid[gridwidth=0.3pt,griddots=5,subgriddiv=1,gridlabels=5pt](-2,-3)(2,2)
\end{pspicture}}

\newcommand{\quadricFanInfty}{%
\psset{xunit=0.4cm,yunit=.4cm}
\begin{pspicture}(-3.2,-3.2)(3.2,3.2)%
\pspolygon[linewidth=.001pt,linecolor=white,fillstyle=solid,fillcolor=gray1](2,2)(.5,.5)(2,-1)%
\pspolygon[linewidth=.001pt,linecolor=white,fillstyle=solid,fillcolor=gray3](2,2)(.5,.5)(-1,2)%
\pspolygon[linewidth=.001pt,linecolor=white,fillstyle=solid,fillcolor=gray5](-1,2)(.5,.5)(-1,-1)%
\pspolygon[linewidth=.001pt,linecolor=white,fillstyle=solid,fillcolor=gray4](-1,-1)(.5,.5)(2,-1)%

\psset{linewidth=1pt}%
\psline{-}(2,2)(.5,.5)(2,-1)%
\psline{-}(-1,2)(.5,.5)(-1,-1)%
\psgrid[gridwidth=0.3pt,griddots=5,subgriddiv=1,gridlabels=5pt](-1,-1)(2,2)
\end{pspicture}}

\newcommand{\tailfanQuadric}{%
\psset{xunit=0.3cm,yunit=.3cm}
\begin{pspicture}(-3.2,-3.2)(3.2,3.2)%
\pspolygon[linewidth=.001pt,linecolor=white,fillstyle=solid,fillcolor=gray1](3,3)(0,0)(3,-3)%
\pspolygon[linewidth=.001pt,linecolor=white,fillstyle=solid,fillcolor=gray3](3,3)(0,0)(-3,3)%
\pspolygon[linewidth=.001pt,linecolor=white,fillstyle=solid,fillcolor=gray5](-3,3)(0,0)(-3,-3)%
\pspolygon[linewidth=.001pt,linecolor=white,fillstyle=solid,fillcolor=gray4](-3,-3)(0,0)(3,-3)%

\psset{linewidth=1pt}%
\psline{-}(3,3)(-3,-3)%
\psline{-}(3,-3)(-3,3)%
\psgrid[gridwidth=0.3pt,griddots=5,subgriddiv=1,gridlabels=5pt](-3,-3)(3,3)
\end{pspicture}}

\newcommand{\degreeQuadric}{%
\psset{xunit=0.3cm,yunit=.3cm}
\begin{pspicture}(-3.2,-3.2)(3.2,3.2)%
\pspolygon[linewidth=.001pt,linecolor=white,fillstyle=solid,fillcolor=gray1](3,3)(.5,.5)(.5,-.5)(3,-3)%
\pspolygon[linewidth=.001pt,linecolor=white,fillstyle=solid,fillcolor=gray3](3,3)(.5,.5)(-.5,.5)(-3,3)%
\pspolygon[linewidth=.001pt,linecolor=white,fillstyle=solid,fillcolor=gray5](-3,3)(-.5,.5)(-.5,-.5)(-3,-3)%
\pspolygon[linewidth=.001pt,linecolor=white,fillstyle=solid,fillcolor=gray4](-3,-3)(-.5,-.5)(.5,-.5)(3,-3)%
\pspolygon[linewidth=.001pt,linecolor=white,fillstyle=crosshatch,hatchsep=0.1,hatchangle=0](.5,.5)(-.5,.5)(-.5,-.5)(.5,-.5)%

\psset{linewidth=1pt}%
\psline{-}(3,3)(.5,.5)(-.5,.5)(-3,3)%
\psline{-}(-3,-3)(-.5,-.5)(.5,-.5)(3,-3)%
\psline{-}(-.5,-.5)(-.5,.5)%
\psline{-}(.5,.5)(.5,-.5)%
\psgrid[gridwidth=0.3pt,griddots=5,subgriddiv=1,gridlabels=5pt](-3,-3)(3,3)
\end{pspicture}}

\newcommand{\boxQuadric}{%
\psset{xunit=.5cm,yunit=.4cm}
\begin{pspicture}(-3,-3)(3,3)%
\pspolygon[linewidth=1pt,fillstyle=solid,fillcolor=gray5](0,3)(-3,0)(0,-3)(3,0)%
\psgrid[gridwidth=0.3pt,griddots=7,subgriddiv=1,gridlabels=5pt](0,0)(-3,-3)(3,3)%
\end{pspicture}}